\documentclass[a4paper, british, reqno]{amsart}

\usepackage[utf8]{inputenc}
\usepackage[T1]{fontenc}
\usepackage{lmodern}
\usepackage{babel}
\usepackage{csquotes}
\usepackage{fullpage}
\usepackage{parskip}
\usepackage{enumitem}
\usepackage{adjustbox}

\usepackage{mathtools}
\usepackage{amssymb}
\usepackage{tikz-cd}
\usetikzlibrary{calc}
\usepackage[all]{xy}

\usepackage[
backend=biber,
style=mathalphabetic,
maxnames=999, giveninits, useprefix,
maxalphanames=4, minalphanames=3,
doi=false, url=false,
dashed=false
]{biblatex}
\DeclareLabelalphaNameTemplate{
	\namepart[use=true, pre=true, strwidth=1, compound=true]{prefix}
	\namepart[compound=true]{family}
}
\usepackage[colorlinks, citecolor=blue]{hyperref}

\usepackage[capitalise, noabbrev]{cleveref}
\newcommand*{\eqrefathome}[1]{\textnormal{(#1)}}
\creflabelformat{equation}{\eqrefathome{#2#1#3}}
\Crefname{diagram}{Diagram}{Diagrams}
\creflabelformat{diagram}{\eqrefathome{#2#1#3}}

\mathtoolsset{mathic}   

\makeatletter
\newcommand*{\comment}[4][]{%
	\@ifundefined{c@#2}{\newcounter{#2}}{}%
	\begingroup
	\ifblank{#1}{%
		\newcommand*{\header}{\csname the#2\endcsname}%
	}{%
		\newcommand*{\header}{(#1\csname the#2\endcsname)}%
	}%
	\addtocounter{#2}{1}%
	\textnormal{\textcolor{#3}{\header.[#4\@]}}%
	\endgroup
}
\makeatother

\newcommand*{\comp}{\circ}
\newcommand*{\from}{\colon}
\newcommand*{\cosmash}{\diamond}

\newcommand*{\eqdef}{\coloneq}
\newcommand*{\iso}{\cong}

\newcommand*{\placeholder}{\:\cdot\:}
\newcommand*{\adj}{\dashv}
\newcommand*{\vadj}{\rotatebox[origin=c]{-90}{\(\adj\)}}
\newcommand{\noproof}{\qed}

\newcommand*{\opp}[1]{{#1}^{\operatorfont op}}

\newcommand*{\Equiv}{\simeq}
\newcommand*{\mapping}{\colon}

\DeclareMathOperator{\Eq}{Eq}
\DeclareMathOperator{\Ker}{Ker} 
\DeclareMathOperator{\Img}{Im}

\DeclarePairedDelimiterX{\IndArr}[1]{(}{)}{
	\renewcommand*{\and}{,}
	#1
}
\DeclarePairedDelimiterX{\CoindArr}[1]{\langle}{\rangle}{
	\renewcommand*{\and}{,}
	#1
}
\DeclarePairedDelimiterX{\GenRel}[2]{\langle}{\rangle}{#1 \mathrel{\delimsize|} #2}

\newcommand*{\CatSymb}[1]{\mathcal{#1}}
\newcommand*{\CatName}[1]{\textnormal{#1}}
\newcommand*{\CondName}[1]{\textnormal{(#1)}}
\newcommand*{\defn}[1]{\emph{#1}}
\newcommand*{\surname}[1]{\textsc{#1}}

\newcommand*{\C}{\CatSymb{C}}
\newcommand*{\V}{\CatSymb{V}}

\newcommand*{\E}{\CatSymb{E}}

\newcommand*{\Grp}{\CatName{Grp}}
\newcommand*{\PXMod}{\CatName{PXMod}}
\newcommand*{\XMod}{\CatName{XMod}}
\newcommand*{\RG}{\CatName{RG}}
\newcommand*{\Grpd}{\CatName{Grpd}}
\newcommand*{\Cat}{\CatName{Cat}}
\newcommand*{\Set}{\CatName{Set}}
\newcommand*{\SSES}{\CatName{SSES}}
\newcommand*{\SES}{\CatName{SES}}
\newcommand*{\Pt}{\CatName{Pt}}

\newcommand*{\SH}{\CondName{SH}}
\newcommand*{\NH}{\CondName{NH}}
\newcommand*{\LACC}{\CondName{LACC}}
\newcommand*{\PCM}{\CondName{PCM}}
\newcommand*{\PFF}{\CondName{PFF}}

\newtheorem{theorem}{Theorem}[section]
\newtheorem{proposition}[theorem]{Proposition}
\AddToHook{env/proposition/begin}{\crefalias{theorem}{proposition}}
\newtheorem{lemma}[theorem]{Lemma}
\AddToHook{env/lemma/begin}{\crefalias{theorem}{lemma}}
\newtheorem{corollary}[theorem]{Corollary}
\AddToHook{env/corollary/begin}{\crefalias{theorem}{corollary}}

\theoremstyle{definition}

\AddToHook{env/definition/begin}{\crefalias{theorem}{definition}}
\newtheorem{definitions}[theorem]{Definitions}
\AddToHook{env/definitions/begin}{\crefalias{theorem}{definitions}}

\theoremstyle{remark}
\newtheorem{remark}[theorem]{Remark}
\AddToHook{env/remark/begin}{\crefalias{theorem}{remark}}
\newtheorem{example}[theorem]{Example}
\AddToHook{env/example/begin}{\crefalias{theorem}{example}}

\tikzcdset{
	mono/.style=tail
}
\tikzcdset{
	epi/.style=two heads
}
\tikzcdset{
	nmono/.style=Triangle[{open, reversed}]->
}
\tikzcdset{
	repi/.style=-Triangle[open]
}
\tikzcdset{
	squared/.style={sep={#1em,between origins}},
	squared/.default=4.5
}
\tikzcdset{
	induced/.style=dashed
}

\DeclareMathDelimiter\AMSlrcorner{\mathclose}{AMSa}{"79}{AMSa}{"79}

\newcommand*{\pb}[2][.2]{
	\arrow[#2, to path={
		let
		\p0 = ($(\tikztotarget.center)-(\tikztostart.center)$)    
		in
		(\tikztostart.center) -- +(-#1*\y0,#1*\y0) node{\Large{\(\AMSlrcorner\)}}
	}, phantom]
}

\newdir{>>}{{}*!/3.5pt/:(1,-.2)@^{>}*!/3.5pt/:(1,+.2)@_{>}*!/7pt/:(1,-.2)@^{>}*!/7pt/:(1,+.2)@_{>}} 
\newdir{ >>}{{}*!/8pt/@{|}*!/3.5pt/:(1,-.2)@^{>}*!/3.5pt/:(1,+.2)@_{>}}
\newdir{ |>}{{}*!/-3.5pt/@{|}*!/-8pt/:(1,-.2)@^{>}*!/-8pt/:(1,+.2)@_{>}} 
\newdir{ >}{{}*!/-8pt/@{>}}
\newdir{>}{{}*:(1,-.2)@^{>}*:(1,+.2)@_{>}}
\newdir{<}{{}*:(1,+.2)@^{<}*:(1,-.2)@_{<}}

\def\pushout{%
	\ar@{-}[]+L+<-6pt,1pt>;[]+LU+<-6pt,6pt>%
	\ar@{-}[]+U+<-1pt,6pt>;[]+LU+<-6pt,6pt>}

\begin{document}

\title{Kaluzhnin--Krasner embedding of precrossed modules}

\author{Maxime \surname{Culot}}
\author{Bo~Shan \surname{Deval}}

\email{maxime.culot@uclouvain.be}
\email{bo.deval@uclouvain.be}

\address[Maxime \surname{Culot}, Bo~Shan \surname{Deval}]{\foreignlanguage{french}{Institut de Recherche en Mathématique et Physique, Université catholique de Louvain, Chemin du Cyclotron~2 bte~L7.01.02, B--1348 Louvain-la-Neuve}, Belgium}    

\thanks{As a Ph.D.\ student, the second author was funded by \emph{Formation à la Recherche dans l'Industrie et dans l'Agriculture (FRIA)}}

\begin{abstract}
	The \emph{Kaluzhnin--Krasner embedding} establishes that, for a group extension, the middle group can be embedded into the wreath product of its kernel and cokernel. Recently, this construction has been generalised in a categorical framework, recovering both the classical group-theoretic result and its analogue for Lie algebras.

	In this article, we apply this categorical framework to two specific cases: precrossed modules and crossed modules (over groups). For precrossed modules, we introduce a Kaluzhnin--Krasner embedding---previously unformulated---by rigorously following the steps of the categorical procedure. In contrast, for crossed modules, we encounter substantial obstacles: while we present partial positive results, we also explain why constructing a full-fledged Kaluzhnin--Krasner embedding in this context is far more difficult.
\end{abstract}

\subjclass[2020]{16B50, 18D15, 18E13, 18G45, 20E22}
\keywords{Wreath product, (split) extension, locally algebraically cartesian closed category, precrossed module, crossed module, semi-abelian category}

\maketitle

\tableofcontents

\section*{Introduction}
It is well known that, given two groups \(A\) and \(X\), any group \(G\) viewed as an extension from \(A\) to \(X\), i.e.\ in the middle of a short exact sequence
\[
	\begin{tikzcd}
		0 \arrow[r] & A \arrow[r, "k", nmono] & G \arrow[r, "f", repi] & X \arrow[r] & 0\text{,}
	\end{tikzcd}
\]
can be embedded into the \emph{wreath product} \(A\wr X\). This latter group is the semi-direct product \(\Set(X,A)\rtimes X\), where \(\Set(X,A)\) is the set of functions from \(X\) to \(A\) equipped with the pointwise multiplication (\((h\cdot h')(x)\coloneq h(x)\cdot h'(x)\) for \(h\), \(h'\in \Set(X,A)\)) and the group action of \(X\) on \(\Set(X,A)\) is given by \(h^{x'}(x)\coloneq h(x\cdot x')\) for \(h\in \Set(X,A)\) and \(x\), \(x'\in X\).\footnote{Throughout the article, following~\cite{DGMVdL}, we write the action of an element \(x\) on an element \(a\) as a superscript \(a^{x}\). Note that this is notation for a \emph{left} action: \((a^{x})^{x'}=a^{x'x}\).} This embedding, called the \emph{Kaluzhnin--Krasner embedding}~\cite{KK51}, is given by
\[
	\phi_G\colon G\to A\wr X\mapping g\mapsto (h_g,f(g))
\]
where \(h_g\colon X\to A\mapping x\mapsto \theta(x)\cdot g\cdot \theta(x\cdot f(g))^{-1}\) for a chosen set-theoretical section \(\theta\) of \(f\).

Since this result has been extended to Lie algebras~\cite{PeRaSh} and to cocommutative Hopf algebras~\cite{BaSiTr14}, it seems that this construction may be expressed in a universal way. This led the authors of~\cite{DGMVdL} to examine in detail the categorical structures arising from the Kaluzhnin--Krasner embedding of groups, with the goal of recovering both the group and Lie algebra cases using a categorical language.

In short, we may state the Kaluzhnin--Krasner embedding in the category \(\Grp\) of groups as an adjunction \begin{tikzcd}[cramped]
	\SES_X(\Grp) \arrow[r, "U", shift left=2] \arrow[r, "\vadj", phantom] & \Grp \arrow[l, "W", shift left=2]
\end{tikzcd} where \(U\) is the forgetful functor that sends an extension over the group \(X\) to its kernel object. In order to generalise this result, the authors of \cite{DGMVdL} investigated the chain of adjunctions
\[
	\begin{tikzcd}
		\SES_X(\Grp) \arrow[r, "L", shift left=2] \arrow[r, "\vadj", phantom] & \SSES_X(\Grp) \arrow[r, "K", shift left=2] \arrow[r, "\vadj", phantom] \arrow[l, "P", shift left=2] & \Grp \arrow[l, "R", shift left=2]
	\end{tikzcd}
\]
where \(\SSES_X(\Grp)\) denotes the category of split extensions over \(X\) in \(\Grp\), \(L\) is the left adjoint to the forgetful functor \(P\) and \(R\) is the right adjoint to the \emph{kernel functor \(K\)} (see~\cite[Section~3]{DGMVdL}).

Even if the existence of \(L\) does not require any further assumption on the underlying category, the existence of the adjoint \(R\) does: we need to consider a \emph{locally algebraically cartesian closed (\LACC{}) category}~\cite{Gra10}. Actually all the known examples for a Kaluzhnin--Krasner embedding satisfy this requirement: \(\Grp\) is \LACC{} since it is the category of internal groups over \(\Set\), a cartesian closed category with pullbacks~\cite[Proposition~5.3]{Gra12}, so is the category of cocommutative Hopf algebras over a field \(\mathbb{K}\) for the same reason\footnote{It is well known that the category of cocommutative Hopf algebras over a field \(\mathbb{K}\) is equivalent to the category of internal objects in the category of cocommutative coalgebras over \(\mathbb{K}\), which is a finitely complete cartesian closed category~\cite{Ba74,GM-VdL18}.} and, finally, the category of Lie algebras is \LACC{} as well~\cite{Gray12b}.

Such a context is quite restrictive among semi-abelian categories (in the sense of Janelidze--Márki--Tholen~\cite{JMT02}). For instance, among all \emph{varieties of non-associative algebras over an infinite field \(\mathbb{K}\) of characteristic different from \(2\)}, the category of Lie algebras is the only one to be \LACC{}~\cite{GVdL19,GVdL19b}. However, being a \LACC{} category does not guarantee an explicit description of a Kaluzhnin--Krasner embedding: as expressed in~\cite[Section~4]{DGMVdL}, we need to compute explicitly the kernel of the induced arrow \(\CoindArr{f\and 1_X}\colon G+X\to X\) case-by-case.

In this article, we establish a Kaluzhnin--Krasner embedding for the category \(\PXMod\) of precrossed modules (over groups), which is \LACC{} as the category of groups internal to the category \(\RG\) of reflexive graphs of sets. Namely, given an extension of precrossed modules
\[
	\begin{tikzcd}
		0 \arrow[r] & A \arrow[r, "k", nmono] & G \arrow[r, "f", repi] & X \arrow[r] & 0\text{,}
	\end{tikzcd}
\]
we want to embed \(G\) into a ``\emph{wreath product}'' in \(\PXMod\).

More precisely, such a result is decomposed into several intermediary results. In \cref{subsec - definitions}, we recall some basic results and characterisations of precrossed modules over \(\Grp\) or a general semi-abelian category. In particular, we consider a precrossed module as a group equipped with two endomorphisms (see \cref{lemma - RG(V) variety}). This observation is crucial to easily compute the kernel of \(\CoindArr{f\and 1_X}\) in \(\PXMod\) from the characterisation of the same kernel in \(\Grp\) (see \cref{subsect - Ker in Grp} and \cref{subsec - kernels in PXMod and XMod}).

With this computation of the kernel, we can apply all the universal machinery developed in \cite{DGMVdL} to obtain the embedding in \(\PXMod\). In \cref{subsec - KK in groups}, we recall the main steps for \(\Grp\). This is crucial to understand all the steps that we will pass through for \(\PXMod\): the description of the power objects in \(\RG\) (see \cref{subsubsec - power object in PXMod}), how we can construct the right adjoint \(R\colon \PXMod\to \SSES_X(\PXMod)\) for a given precrossed module \(X\) by applying Gray's construction~\cite[Proposition~5.3]{Gra12}, and the definition of the morphism \(\chi\colon \Ker(\CoindArr{f\and 1_X})\to A\), which requires a suitable section of the precrossed module morphism \(f\) (see \cref{sec - suitable sections in PXMod}). Finally, in \cref{subsubsection - conclusion in PXMod}, we sum up all those results to apply in \(\PXMod\) the same procedure as in \(\Grp\).

The last part of the document is dedicated to the category \(\XMod\) of crossed modules (over groups) which is also \LACC{} for the same reason as \(\Grp\) and \(\PXMod\). In \cref{subsection - general results about XMod}, we recall the adjunction between \(\PXMod\) and \(\XMod\), which admits an easy description since \(\Grp\) satisfies the ``\emph{Smith is Huq}''~\cite{MVdL12} and ``\emph{Normality of Higgins}''~\cite{AlanThesis} conditions. As a result, we can use our knowledge of \(\Ker(\CoindArr{f\and 1_X})_{\PXMod}\) to describe \(\Ker(\CoindArr{f\and 1_X})_{\XMod}\) (see \cref{subsubsection KLE in XMod}). However, even if the power object in \(\Cat\) is better known than the one in \(\RG\) (see \cref{lem:exp PXMod}), in \cref{sec - chi problem with XMod}, we point out a problem that we cannot solve: since the underlying group in the kernel of \(\CoindArr{f\and 1_X}\) in \(\XMod\) (see \cref{thm - kernel in XMod final version}) is a quotient of the kernel in \(\Grp\) and \(\PXMod\), the usual definition of the morphism \(\chi\) used in \(\Grp\) and in \(\PXMod\) does not define a morphism in \(\XMod\), which stops us from continuing the construction of the Kaluzhnin--Krasner embedding as in those two previous cases.

\section{Internal precrossed modules}
\label{subsec - definitions}
The notion of \emph{crossed modules} was first introduced by Whitehead in~\cite{W49} to study certain properties of the boundary map between the first two homotopy groups of a given pair of pointed topological spaces. The main subjects of interest in the present article are \emph{precrossed modules} and their category \(\PXMod\), as well as its subcategory \(\XMod\) of crossed modules determined by the so-called \emph{Peiffer condition} (see below).

In~\cite{Jan03}, G.~Janelidze introduced the definitions of \defn{internal precrossed modules} and \defn{internal crossed modules} in the semi-abelian context to extend these notions to any semi-abelian category~\cite{JMT02}, the classical ones corresponding to the ones internal to the category \(\Grp\) of groups.

\begin{definitions}\label{def - (pre)crossed module}
	Let \(\C\) be a semi-abelian category and consider two objects \(A\) and \(B\) in \(\C\). We write \(B\flat A\coloneq \Ker(\CoindArr{1_B \and 0}\colon B+A\to B)\) and denote its inclusion in \(B+A\) as \(\kappa_{B,A}\coloneq \ker(\CoindArr{1_B\and 0})\).

	An \defn{internal precrossed module} in \(\C\) is given by a couple of objects \(A\) and \(B\), a morphism \(\partial \colon A\to B\) and an internal action \(\xi\colon B\flat A\to A\) (in the sense of~\cite{BJK05}) satisfying the \emph{Precrossed module condition \PCM{}}, i.e.\ such that the square
	\[
		\tag*{\PCM{}}
		\begin{tikzcd}
			B\flat A\arrow[d, "\xi" '] \arrow[r, "\kappa_{B,A}"] & B+A\arrow[d, "\CoindArr{1_B\and \partial}"] \\
			A\arrow[r, "\partial" ']                             & B
		\end{tikzcd}
	\]
	commutes.

	An internal precrossed module \((A,B,\xi,\partial)\) is an \defn{internal crossed module} whenever the \emph{Peiffer condition \PFF{}} is verified, i.e.\ when the square
	\[
		\tag*{\PFF{}}
		\begin{tikzcd}[column sep=large]
			(B+A)\flat A \arrow[r,"\CoindArr{1_B\and \partial}\flat 1_A"] \arrow[d, "\underline{\CoindArr{1_B\and \iota_2}}" '] & B\flat A\arrow[d,"\xi"] \\
			B\flat A \arrow[r,"\xi" ']                                                                                          & A
		\end{tikzcd}
	\]
	commutes, where \(\underline{\CoindArr{1_B\and \iota_2}}\) is the unique morphism such that \(\kappa_{B,A}\underline{\CoindArr{1_B\and \iota_2}}=\CoindArr{1_B\and \iota_2} \kappa_{B+A, A}\)  and \(\iota_2\) is the second inclusion of the binary coproduct \(B+A\).

	A \defn{morphism of internal (pre)crossed modules} \((f_A,f_B)\colon (A,B,\xi,\partial)\to (A',B',\xi',\partial')\) is a couple of morphisms \(f_A\colon A\to A'\) and \(f_B\colon B\to B'\) in \(\C\) such that \(f_B\partial = \partial' f_A\) and \(f_A \xi = \xi ' (f_B\flat f_A)\).

	The above data describe the \defn{category of internal precrossed modules in \(\C\)} and the \defn{category of internal crossed modules in \(\C\)}, denoted respectively by \(\PXMod(\C)\) and \(\XMod(\C)\).
\end{definitions}

These definitions are not always easy to manipulate but, fortunately, there is an alternative point of view we may adopt depending on the situation. The category \(\PXMod(\C)\) is equivalent to the category \(\RG(\C)\)\label{def - usual def RG} of internal reflexive graphs in \(\C\): an \defn{internal reflexive graph} in \(\C\) is a diagram \(\begin{tikzcd}[cramped] G_1 \arrow[r, "d", shift left=2] \arrow[r, "c"', shift right=2] & G_0 \arrow[l, "\iota" description]\end{tikzcd}\) where \(G_1\) and \(G_0\) are objects in \(\C\) and \(d\), \(c\), \(\iota\) are morphisms of \(\C\) such that \(d\iota=1_{G_0}=c\iota\). We write \((G_0,G_1, d,c, \iota)\) to design this reflexive graph. A \defn{morphism of internal reflexive graphs} \((f_0,f_1)\colon (G_0,G_1, d,c, \iota)\to (G_0',G_1', d',c', \iota')\) is a couple of morphisms \(f_0\colon G_0\to G_0'\), \(f_1\colon G_1\to G_1'\) of \(\C\) such that all the squares in
\[
	\begin{tikzcd}
		G_1 \arrow[r, "d", shift left=2.5] \arrow[r, "c"', shift right=2.5] \arrow[d, "f_1"'] & G_0 \arrow[l, "\iota" description] \arrow[d, "f_0"] \\
		G_1' \arrow[r, "d'", shift left=2.5] \arrow[r, "c'"', shift right=2.5]                & G_0' \arrow[l, "\iota'" description]
	\end{tikzcd}
\]
commute.

If \(\C\) is a variety of \(\Omega\)-groups\footnote{An \(\Omega\)-group is a group \(G\) equipped with a set \(\Omega\) of additional operations of positive arity such that, for all \(\omega\in \Omega\), \(\omega(e,\dots, e)=e\) where \(e\) is the neutral element of \(G\). A morphism of \(\Omega\)-groups is a group morphism \(f\from G\to G'\) compatible with every \(\omega\in \Omega\): \(f(\omega(g_1,\dots, g_n))=\omega(f(g_1),\dots, f(g_n))\).}~\cite{Hig56}, reflexive graphs may also be viewed as a variety of \(\Omega\)-groups, which will be the primary point of view we will adopt in this paper. In the following, we will freely use the notation \(\PXMod(\C)\) to talk about precrossed modules in the original sense, the reflexive graph one or the \(\Omega\)-group one. Likewise, \(\XMod(\C)\) denotes the crossed module subcategory of \(\PXMod(\C)\) in any of these three points of view.

\begin{lemma}[see proof of {\cite[Proposition~2.3]{GR04}}]
	\label{lemma - RG(V) variety}
	Let \(\C\) be a regular category. Then the category \(\RG(\C)\) of reflexive graphs in \(\C\) is equivalent to the following category:
	\begin{itemize}
		\item its objects are triplets \((X,s,t)\) where \(s\), \(t\from X\to X\) are endomorphisms of \(X\) such that \(ts=s\) and \(st=t\),
		\item an arrow from \((X,s,t)\) to \((X',s',t')\) is an arrow \(f\from X\to X'\) in \(\C\) satisfying \(fs=s'f\) and \(ft=t'f\).
	\end{itemize}
\end{lemma}

\begin{proof}
	Let us consider an internal reflexive graph \(\begin{tikzcd}[cramped]
		G_1 \arrow[r, "d", shift left=2] \arrow[r, "c"', shift right=2] & G_0 \arrow[l, "\iota" description]
	\end{tikzcd}\) in \(\C\). By setting \(s\eqdef \iota d\from G_1\to G_1\) and \(t\eqdef \iota c\from G_1 \to G_1\), then the equations
	\begin{equation}
		\label{eq - s and t}
		st=t \qquad\text{and}\qquad ts=s
	\end{equation}
	hold.

	For the converse, it suffices to consider the factorisation images of \(s\) and \(t\). By the equations \eqref{eq - s and t}, we can prove that \(\Img(s)\) is isomorphic to \(\Img(t)\) and the monic parts of the factorisation are the same up to isomorphism. Therefore, it suffices to take this monomorphism as the morphism \(\iota\) and the two regular epimorphisms as \(d\) and \(c\).
\end{proof}

\begin{remark}
	\label{rem:s and t idempotent}
	From the equations~\eqref{eq - s and t}, we also have that \(s\) and \(t\) are idempotent:
	\[
		ss=s(ts)=ts=s \qquad\text{and}\qquad tt=t(st)=st=t\text{.}
	\]
\end{remark}

\begin{remark}[Expression of \cref{lemma - RG(V) variety} in groups]
	Let us recall that a precrossed module in the sense of \cref{def - (pre)crossed module} is nothing else than two groups \(A\) and \(B\), a morphism \(\partial\colon A\to B\) and a group action of \(B\) on \(A\) (denoted \(a^b\)) satisfying
	\[
		\partial\left(a^b\right)=b\partial(a)b^{-1}
	\]
	for all \(a\in A\) and \(b\in B\). The group action determines the split extension
	\[
		\begin{tikzcd}[cramped]
			0 \arrow[r] & A \arrow[r, "k", nmono] & A\rtimes B \arrow[r, "d", repi, shift left] & B \arrow[l, "\iota", mono, shift left] \arrow[r] & 0
		\end{tikzcd}
	\]
	where \(A\rtimes B\) is the semi-direct product---the set \(A\times B\) equipped with the multiplication \((a,b)\cdot (a',b')=(a\cdot a'^b,bb')\)---while \(k(a)=(a,e_B)\), \(\iota(b)=(e_A,b)\) and \(d(a,b)=b\) for \(a\in A\) and \(b\in B\).

	By the above equation, the right-hand side of the split extension turns out to be a reflexive graph where \(c\colon A\rtimes B\to B\) is defined as \(c(a,b) \eqdef \partial(a)b\).

	As a consequence, we can apply the proof of \cref{lemma - RG(V) variety} to express the precrossed module as the group \(A\rtimes B\) with \(s(a,b) \eqdef (e_A,b)\) and \(t(a,b) \eqdef (e_A,\partial (a)b)\).
\end{remark}

From the previous lemma, we can deduce that \(\PXMod(\Grp)\) (or shortly \(\PXMod\)) is a variety of \(\Omega\)-groups or, even more precisely, of \defn{distributive} \(\Omega\)-groups~\cite{Hig56}, i.e.\ of \(\Omega\)-groups such that every \(\omega\in \Omega\) is distributive with respect to the group multiplication \(\star\): for all \(g_1\), \dots, \(g_n\), \(h\in G\) and for all \(i \in \{1, \dots, n\}\),
\begin{equation}
	\label{eq - distributive omega group}
	\omega(g_1,\dots , g_i\star h,\dots, g_n)=\omega(g_1,\dots, g_i,\dots, g_n)\star \omega(g_1,\dots, h,\dots, g_n)\text{.}
\end{equation}
In particular, if \(\omega\in \Omega\) is a unary operation, like \(s\) and \(t\) in our case, then this condition~\eqref{eq - distributive omega group} is equivalent to requiring that \(\omega\) is a group morphism.

\begin{corollary}
	\label{cor:PXMod(V) omega distributive if V is so}
	For a variety \(\V\) of \(\Omega\)-groups, \(\PXMod(\V)\) is also a variety of \(\Omega\)-groups. If in addition, \(\V\) is distributive, then so is \(\PXMod(\V)\).
\end{corollary}

\begin{proof}
	The first statement about \(\Omega\)-groups follows from the equivalence \(\PXMod(\V) \Equiv \RG(\V)\) since any variety \(\V\) of \(\Omega\)-groups is a semi-abelian variety: using the underlying group structure, we can prove that it is a protomodular variety~\cite{BJ2003} and thus a semi-abelian category~\cite{JMT02}. Therefore, \cref{lemma - RG(V) variety} tells us that taking precrossed modules is the same as simply adding two endomorphisms \(s\), \(t\) to the \(\Omega\)-group operations (\(s\) and \(t\) are compatible as wanted since they are morphisms in \(\V\)). Finally, if we start with distributive \(\Omega\)-groups, we keep the distributivity since the new (unary) operations \(s\) and \(t\) are in particular group morphisms.
\end{proof}

Now, here is a description of coproducts of internal precrossed modules which will be useful later.

\begin{lemma}
	\label{lemma:adding unary oper}
	Let \(\C\) be a category and \(J\) be a set. Consider the following category \(\C_{\textnormal{End}}\):
	\begin{itemize}
		\item its objects are couples \(\bigl(C , (t_j)_{j \in J}\bigr)\) where \(C\) is an object of \(\C\) and \(t_j\), \(j \in J\), are endomorphisms of \(C\) satisfying a certain set of equations involving only compositions of these \(t_j\)'s,
		\item an arrow from \(\bigl(C , (t_j)_{j \in J}\bigr)\) to \(\bigl(C' , (t'_j)_{j \in J}\bigr)\) in \(\C_{\textnormal{End}}\) is an arrow \(f \from C \to C'\) in \(\C\) such that \(f t_j = t'_j f\) for every \(j \in J\).
	\end{itemize}
	If we have two objects \(\bigl(A , (a_j)_{j \in J}\bigr)\), \(\bigl(B , (b_j)_{j \in J}\bigr)\) of \(\C_{\textnormal{End}}\) that have a binary product \(\bigl(A \times B , p_1 , p_2\bigr)\) in \(\C\), then their binary product in \(\C_{\textnormal{End}}\) exists and is given by the object \(\bigl(A \times B , (a_j \times b_j)_{j \in J}\bigr)\) with the same projections.
\end{lemma}

\begin{proof}

	First, we check that \(\bigl(A \times B , (a_j \times b_j)_{j \in J}\bigr)\) is an object of \(\C_{\textnormal{End}}\) since composition is done componentwise (i.e.\ \((f \times g)  (f' \times g') = (f  f') \times (g  g')\)). For the projections \(p_1 \from A \times B \to A\) and \(p_2 \from A \times B \to B\), we compute \(p_1  (a_j \times b_j) = a_j  p_1\) for all \(j \in J\) as wanted and similarly for \(p_2\) so they are indeed arrows in \(\C_{\textnormal{End}}\).

	Finally, given another span \begin{tikzcd}[cramped]
		\bigl(A , (a_j)_{j \in J}\bigr) & \bigl(C , (c_j)_{j \in J}\bigr) \arrow[l, "f"'] \arrow[r, "g"] & \bigl(B , (b_j)_{j \in J}\bigr)
	\end{tikzcd}, the universal property of \(A \times B\) in \(\C\) gives us a unique arrow \(\IndArr{f \and g} \from C \to A \times B\) making
	\[
		\begin{tikzcd}
			& C \arrow[ld, "f"'] \arrow[d, "\IndArr{f \and g}" description, induced] \arrow[rd, "g"]     \\
			A & A \times B \arrow[l, "p_1"] \arrow[r, "p_2"']                                          & B
		\end{tikzcd}
	\]
	commute and we must verify that \(\IndArr{f \and g}\) is an arrow in \(\C_{\textnormal{End}}\) to finish the proof, which is done by postcomposing \(p_1\) and \(p_2\) to the equations to verify. For example, we have
	\[
		p_1  \bigl(\IndArr{f \and g}  c_j\bigr) = f  c_j
	\]
	and
	\[
		p_1  \bigl((a_j \times b_j)  \IndArr{f \and g}\bigr) = (a_j  p_1)  \IndArr{f \and g} = a_j  f
	\]
	which are equal since \(f\) is an arrow in \(\C_{\textnormal{End}}\).
\end{proof}

\begin{remark}
	In this article, we only focus on binary coproducts (so the dual of the above statement). However, the above result and the associated proof may be extended from the binary case to products of arbitrary size.
\end{remark}

\begin{corollary}
	\label{cor: coproducts in PXMod}
	In a semi-abelian category \(\C\), the binary coproducts of internal precrossed modules are built in \(\C\) with their endomorphisms being the induced morphisms.\noproof
\end{corollary}

Therefore, the binary coproduct of two precrossed modules \((G_1, s_1, t_1)\) and \((G_2, s_2, t_2)\) is given by the object \((G_1+_{\C} G_2, s_1+_{\C} s_2, t_1+_{\C} t_2)\) with its two inclusions being the ones from the binary coproduct in \(\C\). To simplify the notation, we will often omit the subscript \(\C\) if no confusion is possible: we write \((G_1+ G_2, s_1+ s_2, t_1+ t_2)\).

\begin{remark}
	\label{rem - binary coproducts usual definition}
	If we consider the classical approach (which will be used in \cref{subsec - kernels in PXMod and XMod}) as in p.~\pageref{def - usual def RG}, then we can reformulate the previous results as: given two internal precrossed modules \((X_0	,X_1, d^X, c^X, \iota^X)\) and \((G_0,G_1,d^G,c^G,\iota^G)\) in a semi-abelian category, their binary coproduct is given by
	\[
		(G_0+X_0, G_1+X_1, d^X+d^G, c^X+c^G, \iota^X+\iota^G)
	\]
	where the objects and the morphisms are induced by the binary coproducts in the underlying category.
\end{remark}

\section{A first step toward embedding theorems: computation of kernels}\label{sec - embedding theorem}
\subsection{Description of the kernel of \texorpdfstring{\(\CoindArr{f \and 1_X}\)}{⟨f,1\_X⟩} in groups}\label{subsect - Ker in Grp}
In~\cite{CRVdL}, the authors described the kernel of \(\CoindArr{1_X \and 1_X}\) in groups for a given group \(X\) as the subgroup of the free group \(X+X\) generated by elements of the shape
\[
	\overline{x}\underline{x}^{-1} \quad\text{and}\quad \underline{y}\overline{y}^{-1}
\]
for \(x\), \(y\in X\), where we write \(\overline{(\placeholder)}\) and \(\underline{(\placeholder)}\) for the elements which belong, respectively, to the first or second copy of \(X\) inside the coproduct \(X+X\).

However, in~\cite{DGMVdL}, the authors enlarge the previous case where they prove that, in groups, the kernel of \(\CoindArr{f\and 1_X}\), for a group morphism \(f\from G\to X\), is the group admitting the presentation \(\GenRel{S}{R}\) where \(S=G\times X\) and
\begin{equation}
	\label{eq: relation for kernel in group}
	R = \{
	(e_G,x) = e \mid x\in X\} \cup \{(g,x)(g',xf(g))=(gg',x)\mid g,g'\in G, x\in X\}
	\text{.}
\end{equation}

The idea of this presentation is to see an element \((g,x)\) as the word \(xgf(g)^{-1}x^{-1}\) in the free product \(G+X\). This idea helps to compare the two expressions.

\begin{remark}
	\label{rem: the kernel described for group is a normal subgroup}
	Of course, the presentation \(\GenRel{S}{R}\) is not a proper subgroup of \(G+X\) but it is isomorphic to the kernel. As a kernel, we already know that the image of this presentation in \(G+X\) is a normal subgroup but we can also check it by hand. This image is the subgroup of \(G+X\) generated by the elements of the form \(xgf(g)^{-1}x^{-1}\), let us check its normality. Given \(y\in X\) and \(h\in G\), we have
	\[
		\begin{split}
			yxgf(g)^{-1}x^{-1}y^{-1} & =(yx)gf(g)^{-1}(yx)^{-1}\text{;}                                                     \\
			hxgf(g)^{-1}x^{-1}h^{-1} & =(hf(h)^{-1})\bigl((f(h)x)gf(g)^{-1}(f(h)x)^{-1}\bigr)\bigl(f(h)h^{-1}\bigr)\text{.}
		\end{split}
	\]
\end{remark}

The two descriptions differ in the number of factors needed to write down a generator. In the second one, the generator attached to a couple \((g,x)\in G\times X\) is the word \(xgf(g)^{-1}x^{-1}\), that is, the conjugate by \(x\in X\) of the element \(gf(g)^{-1}\); it is a product of the four factors \(x\), \(g\), \(f(g)^{-1}\) and \(x^{-1}\). By contrast, the generators \(\overline{x}\underline{x}^{-1}\) and \(\underline{y}\overline{y}^{-1}\) of the first description are products of only two factors. This is no coincidence: the first description is the special case \(G=X\) and \(f=1_X\) of the second, and there the four factors collapse into two. Indeed, \(f=1_X\) is then surjective, so given \(x\in X\) and \(g\in G\) we may choose \(g'\in G\) with \(f(g')=x\) and obtain
\[
	xgf(g)^{-1}x^{-1}=\bigl(x(g')^{-1}\bigr)\bigl(g'gf(g)^{-1}f(g')^{-1}\bigr)= \bigl(f(g')(g')^{-1}\bigr)\bigl(g'gf(g'g)^{-1}\bigr)\text{,}
\]
where the first bracket is of the shape \(\underline{y}\overline{y}^{-1}\) and the second of the shape \(\overline{x}\underline{x}^{-1}\), so that each is a generator of the first description.

\subsection{Description in \texorpdfstring{\(\PXMod\)}{PXMod}}
\label{subsec - kernels in PXMod and XMod}
The computation of the kernel of the induced morphism \(\CoindArr{f\and 1_X}\) will be straightforward for \(\PXMod\) since on the one hand, we already know what the kernel is for groups and, on the other hand, they are both distributive \(\Omega\)-groups (see \cref{cor:PXMod(V) omega distributive if V is so}).

\begin{lemma}
	\label{lemma: kernel in PXMod}
	Let \((G,s_G,t_G)\) and \((X,s_X,t_X)\) be two precrossed modules, and consider their binary coproduct \((G+X,s_G+s_X,t_G+t_X)\) (see \cref{cor: coproducts in PXMod}). For a morphism \(f\colon G\to X\) of precrossed modules, the kernel in \(\PXMod\) of \(\CoindArr{f\and 1_X}\) is the precrossed module \((\Ker(\CoindArr{f\and 1_X})_\Grp, s_{\PXMod}, t_{\PXMod})\) where \(\Ker(\CoindArr{f\and 1_X})_\Grp\) denotes the kernel as computed in groups and where \(s_{\PXMod}\) and \(t_{\PXMod}\) denote, respectively, the restriction to \(\Ker(\CoindArr{f\and 1_X})_\Grp\) of the operations \(s_G+s_X\) and \(t_G+t_X\).
\end{lemma}

\begin{proof}
	Since \(\CoindArr{f\and 1_X}\) is the same morphism as in groups, its kernel in \(\PXMod\) has the same underlying set as the kernel in \(\Grp\) and its precrossed module structure is the one inherited from the coproduct \((X+G,s_X+s_G,t_X+t_G)\).
\end{proof}

Moreover, we have the following result about kernels of morphisms in a given variety of distributive \(\Omega\)-groups:
\begin{theorem}[\cite{Hig56}]
	\label{thm - Higgins kernel distributive omega groups}
	If \(G\) is a distributive \(\Omega\)-group and \(H\) is an \(\Omega\)-subgroup of \(G\), then \(H\) is a kernel if and only if \(H\) is a normal subgroup of \(G\) and, for all \(g_1,\dots, g_n\in G\), \(h\in H\) and \(\omega\in \Omega\), \(\omega(g_1,\dots, g_{i-1}, h, g_{i+1},\dots, g_n)\in H\).\noproof
\end{theorem}

In our setting, \(\Omega=\{s,t\}\) and therefore, the \(\Omega\)-operations in the binary coproducts of two reflexive graphs \((G,s_G,t_G)\) and \((X,s_X,t_X)\) are \(s_X+s_G\) and \(t_X+t_G\). If we combine this observation with the previous theorem, we obtain:

\begin{remark}[Normality via distributive \(\Omega\)-groups]
	Using the distributive \(\Omega\)-group structure of \(\PXMod\), we can also check by hand that the precrossed module \((\Ker(\CoindArr{f\and 1_X})_\Grp, s_{\PXMod}, t_{\PXMod})\) describes a kernel in \(\PXMod\).

	By \cref{lemma - RG(V) variety}, we know that any precrossed module in the usual sense can be viewed as a group with two unary operations. This proves that \(\PXMod\) is indeed a variety of distributive \(\Omega\)-groups (see \eqref{eq - distributive omega group}). Hence, we can apply \cref{thm - Higgins kernel distributive omega groups} to prove our claim: it suffices to prove that the generator of the group \(\Ker(\CoindArr{f\and 1_X})_\Grp\) is closed under \(s_{\PXMod}\) and \(t_{\PXMod}\). Since the proof is symmetric, we prove the closure under \(s_{\Ker(\CoindArr{f\and 1_X})}\): given \(x\in X\) and \(g\in G\), we have
	\[
		\begin{split}
			s_{\PXMod}\bigl(xgf(g)^{-1}x^{-1}\bigr) & =s_X(x)s_G(g)s_X(f(g))^{-1}s_X(x)^{-1} \\
			                                        & =s_X(x)s_G(g)f(s_G(g))^{-1}s_X(x)^{-1} \\
			                                        & = yhf(h)^{-1}y^{-1}
		\end{split}
	\]
	where \(s_X(x)=y\in X\) and \(s_G(g)=h\in G\) since \(s_X\) and \(s_G\) are endomorphisms of, respectively, \(X\) and \(G\).
\end{remark}

\section{Kaluzhnin--Krasner embedding theorems}

\subsection{Classical Kaluzhnin--Krasner embedding}\label{subsec - KK in groups}

As explained in~\cite{DGMVdL}, the category \(\XMod\) is a good candidate for a new example of a Kaluzhnin--Krasner embedding theorem. This is because it is a \emph{locally algebraically cartesian closed \LACC{}} category, a condition on the existence of a certain right adjoint.

It is important to point out that this assumption is quite strong. For instance, the only variety of non-associative algebras (over an infinite field of characteristic different from \(2\)) that is \LACC{} is the variety of Lie algebras~\cite{GVdL19,GVdL19b}. This is the reason why we restrict this article to \(\PXMod\) and \(\XMod\). Indeed, it is well known that \(\PXMod\) and \(\XMod\) are, respectively, equivalent to the category of groups internal to reflexive graphs\footnote{As recalled in~\cite{Jan03}, we can prove that \(\PXMod\) is equivalent to the category \(\RG(\Grp)\) of reflexive graphs internal to \(\Grp\), which itself turns out to be equivalent to the category \(\Grp(\RG)\) of groups internal to the category of reflexive graphs.} and internal to small categories~\cite{LR80,MLan98}. This implies that \(\PXMod\) and \(\XMod\) are \LACC{} since any category of internal groups in a cartesian closed category with pullbacks (such as \(\RG\) and \(\Cat\)~\cite{MLan98}) is \LACC{}~\cite[Proposition~5.3]{Gra12}.

The classical Kaluzhnin--Krasner embedding is a result about groups saying:
\begin{theorem}[\cite{KK51}]
	For any group extension
	\(\begin{tikzcd}[cramped, sep=small]
		0 \arrow[r] & A \arrow[r, nmono] & G \arrow[r, repi] & X \arrow[r] & 0
	\end{tikzcd}\),	the group \(G\) can be embedded into the wreath product \(A \wr X\) via a group homomorphism \(\phi_G\colon G\to A \wr X\).
\end{theorem}

To understand this theorem, let us first recall the definition of the \defn{wreath product} of two groups \(A\) and \(X\): it is the semi-direct product \(\Set(X,A)\rtimes X\) where the group multiplication of \(\Set(X,A)\) is pointwise (given \(h\), \(h'\in \Set(X,A)\) and \(x\in X\), we have \((hh')(x) \eqdef h(x)h'(x)\)) and with respect to the canonical \(X\)-action on it (\(h^x(x') \eqdef h(x'x)\) for all \(h\in \Set(X,A)\) and \(x\), \(x'\in X\)).

The authors of~\cite{DGMVdL} look for a universal way to express the above theorem. To sum up, they investigate these different adjunctions for a given \LACC{} semi-abelian category \(\C\)
\[
	\begin{tikzcd}
		\SES_X(\C) \arrow[r, "L", shift left=2] \arrow[r, "\vadj", phantom] & \SSES_X(\C) \arrow[r, "K", shift left=2] \arrow[r, "\vadj", phantom] \arrow[l, "P", shift left=2] & \C \arrow[l, "R", shift left=2]\arrow[ll, "W=PR", bend left, out=45, in=135]
	\end{tikzcd}
\]
where \(\SSES_X(\C)\) is the category of \emph{split short exact sequences} (or \emph{split extensions}) over \(X\) and \(\SES_X(\C)\) is the category of \emph{short exact sequences} (or \emph{extensions}) over \(X\) for a given object \(X\) in \(\C\). If we see a split short exact sequence
\[
	S=\bigl(\begin{tikzcd}[cramped]
			0 \arrow[r] & A \arrow[r, "k", nmono] & G \arrow[r, "f", repi, shift left] & X \arrow[l, "s", mono, shift left] \arrow[r] & 0
		\end{tikzcd}\bigr)
\]
as \((A,G,k,f,s)\) and a short exact sequence
\[
	E=\bigl(\begin{tikzcd}[cramped]
			0 \arrow[r] & A \arrow[r, "k", nmono] & G \arrow[r, "f", repi] & X \arrow[r] & 0
		\end{tikzcd}\bigr)
\]
as \((A,G,k,f)\), then we can understand the above adjunctions as follows: \(K\) sends a split extension \((A,G,k,f,s)\) to its kernel object \(A\) and \(R\) is its right adjoint which exists exactly when \(\C\) is \LACC{}\footnote{The original definition of \LACC{}\@, as stated, by J.~R.~A.~Gray~\cite{Gra10}, is the existence of a right adjoint for any change-of-base functor \(f^\star\colon \Pt_Y(\C)\to \Pt_X(\C)\) for \(f\in \C(X,Y)\). However, if the category \(\C\) is pointed and (Bourn-)protomodular, then it is enough to require that the kernel functor (denoted here by \(K\)) admits a right adjoint~\cite[Theorem~5.1]{Gra12}.}, and \(P\) is the functor that forgets the splitting by mapping \((A,G,k,f,s)\) to \((A,G,k,f)\) and always admits a left adjoint \(L\). The adjunction \(K\adj R\) defines a universal wreath product but this embedding does not express properly the classical theorem since it is only about \emph{split} short exact sequences. This explains the use of the other adjunction \(L\adj P\). However, simply composing these adjunctions would give us a universal extension with respect to \(KL(E)\) and not with respect to the kernel object \(A\) of the starting short exact sequence \(E\) as wanted. If we want to embed this short exact sequence \(E\) from \(A\) over \(X\) into \(PR(A)\) rather than into \(RPKL(E)\), then we need the existence of a monomorphism \(\phi\colon E\to PRU(E)\), where \(U\colon \SES_X(\C)\to \C\) is the forgetful functor that sends an extension to its kernel object, for each \(E\). The authors expressed this last condition by studying the kernel object of
\[
	PL(E)=\bigl(\begin{tikzcd}[cramped]
			0 \arrow[r] & KL(E) \arrow[r, nmono] & G+X \ar[r, "\CoindArr{f\and 1_X}", repi] & X \arrow[r] & 0
		\end{tikzcd}
	\bigr)\text{.}
\]
This explains the link with the previous parts of the article, since \(KL(E)\) is the kernel \(\Ker(\CoindArr{f\and 1_X})\) of which we have constructed an explicit description for \(\PXMod\) (see \cref{lemma: kernel in PXMod}).

\begin{remark}[Categorical approach of the Kaluzhnin--Krasner embedding in \(\Grp\)]
	\label{rem-Gray construction of R}
	In this remark, we spell out the different steps of the categorical approach described in~\cite{DGMVdL} since we will follow the same approach for \(\PXMod\) and \(\XMod\). In addition, we point out where the case-by-case study is important. We start with a short exact sequences of groups encoded as the top sequence of
	\begin{equation}
		\label{eq-sum up of KK in Grp via category}
		\begin{tikzcd}
			\SES_X(\Grp)\arrow[d, "L"']   & 0 \arrow[r] & A \arrow[r,nmono, "k"] \arrow[d, "\lambda_1"', induced] & G \arrow[r, "f", repi] \arrow[d, "\iota_1 = \lambda_2"']                             & X \arrow[r] \arrow[d, equal]                                  & 0 \\
			\SSES_X(\Grp)\arrow[d, "RK"'] & 0 \arrow[r] & KL(E) \arrow[r,nmono] \arrow[d, "\eta_1"', induced]     & G+X \arrow[r, "\CoindArr{f \and 1_X}", shift left, repi] \arrow[d, "\phi = \eta_2"'] & X \arrow[r] \arrow[l, "\iota_2", shift left] \arrow[d, equal] & 0 \\
			\SSES_X(\Grp)\arrow[d, "P"']  & 0 \arrow[r] & KL(E)^X \arrow[r,nmono] \arrow[d, equal]                & KL(E)^X \rtimes X \arrow[r, shift left,repi] \arrow[d,equal]                         & X \arrow[r] \arrow[d, equal]  \arrow[l, shift left]           & 0 \\
			\SES_X(\Grp) \arrow[d]        & 0 \arrow[r] & KL(E)^X \arrow[r,nmono] \arrow[d, induced, "W(\chi)" '] & KL(E)^X \rtimes X \arrow[r,repi] \arrow[d, "W(\chi)\rtimes 1_X"]                     & X \arrow[r] \arrow[d, equal]                                  & 0 \\
			\SES_X(\Grp)                  & 0 \arrow[r] & A^X \arrow[r,nmono]                                     & A^X \rtimes X \arrow[r,repi]                                                         & X \arrow[r]                                                   & 0
		\end{tikzcd}
	\end{equation}
	and we follow these several steps to embed it into the bottom sequence, where
	\begin{itemize}
		\item \(KL(E)^X\) and \(A^X\) are the set-theoretic power objects;
		\item \(\lambda_2 \from G \to G+X\) is the first inclusion into the coproduct and \(\lambda_1\) its restriction between the kernels;
		\item the passage from the second line to the third one corresponds to the adjunction \(K \adj R\): \(\eta_1(a)(x) \eqdef xax^{-1}\) for \(a \in KL(E)\);
		\item combining the two previous points, since the element \(g\) may be decomposed as \(\bigl(gf(g)^{-1}\bigr) f(g)\) with \(gf(g)^{-1} \in KL(E)\) and \(f(g) \in X\), we have the composite \((\phi \comp \iota_1)(g) = (h_g,f(g))\in KL(E)^X\rtimes X\) where \(h_g(x) \eqdef (g,x)\in KL(E)\), using the presentation of \cref{subsect - Ker in Grp};
		\item \(W(\chi)\) is the postcomposition by
		      \[
			      \chi\colon KL(E) \to A \mapping (g,x)\mapsto \theta(x)\cdot g\cdot  \theta(x\cdot f(g))^{-1}
		      \]
		      so the element of the previous point is finally sent to the element \((\chi \comp h_g , f(g)) \in A^X\rtimes X\) where \((\chi \comp h_g)(x) = \theta(x)\cdot g\cdot  \theta(x\cdot f(g))^{-1}\) for \(x \in X\), recovering the known embedding for groups.
	\end{itemize}
	Note that the passage from the second line to the third one express the fact that \(\Grp\) is \LACC{}\@, which is a particular instance of \cite[Proposition~5.3]{Gra12}.

	As we may observe, the only needed specifications for the case of groups are the explicit description of \(KL(E)\) (which we have done for \(\PXMod\) in \cref{lemma: kernel in PXMod}), of \(\chi\colon KL(E)\to A\), which we will be able to construct thanks to the aforementioned description of \(KL(E)\), and finally the description of \(A^X\) for two objects in the category of which we take the group objects---\(\RG\) in the case of \(\PXMod\) and \(\Cat\) for \(\XMod\)---in order to apply Gray's construction~\cite[Proposition~5.3]{Gra12}.
	Concerning this last point, we recall that \(\RG\) (as a presheaf category  over a small category and see for instance~\cite[Theorem~2.3.4]{Bor94d}) and \(\Cat\) (see for instance~\cite{MLan98}) are cartesian closed categories as required.
\end{remark}

\subsection{Kaluzhnin--Krasner embedding of precrossed modules}
\label{sec:KK PXMod}
In this setting, the group underlying \(\Ker(\CoindArr{f\and 1_X})_{\PXMod}\) admits the same group presentation as \(\Ker(\CoindArr{f\and 1_X})_{\Grp}\)---see \eqref{eq: relation for kernel in group}---since they coincide.

\subsubsection{Power object and \LACC{} for \texorpdfstring{\(\PXMod\)}{PXMod}}\label{subsubsec - power object in PXMod}

To apply Gray's construction, using the fact that \(\PXMod\) is equivalent to the category of internal groups in the category \(\RG\) of reflexive graphs of sets, we first need to understand the power object in \(\RG\). More precisely, we will describe this category as a category of presheaves, for which there exists an explicit description of the power object: we can see \(\RG\) as the presheaf category \([\opp{\E}, \Set]\) where \(\E\) is the small category with one object \(\ast\) and two non-identity arrows \(\sigma\) and \(\tau\) whose composites are given by
\begin{align}
	\label{eq:composition RG presheaf}
	\sigma \tau & = \sigma & \tau \sigma & = \tau\text{.}
\end{align}
Indeed, a reflexive graph can be seen as a functor \(F\in  [\opp{\E}, \Set]\), i.e.\ the data of a set \(A = F(\ast)\) equipped with two endomorphisms \(s_A = F(\sigma)\), \(t_A = F(\tau)\) such that
\begin{gather*}
	t_A s_A = F(\tau)\circ F(\sigma) = F(\sigma\tau) = F(\sigma) = s_A \\
	s_A t_A = F(\sigma)\circ F(\tau) = F(\tau\sigma) = F(\tau) = t_A
\end{gather*}
which gives back the usual characterisation as expected (see \cref{lemma - RG(V) variety}).

With this description of \(\RG\) as a presheaf category, we can follow the following construction to find its power object:

\begin{lemma}
	\label{lem:exp PXMod}
	Given two reflexive graphs \(A = (A,s_A,t_A)\) and \(X = (X,s_X,t_X)\), their power object \(A^X\) is given by the set
	\[
		A^X = \RG(\{1,\sigma,\tau\} \times X,A)
	\]
	where the unary endomorphisms on \(\{1,\sigma,\tau\}\) are given by multiplication on the right in \(\E\) by \(\sigma\) and \(\tau\), respectively, following the composition rules of~\eqref{eq:composition RG presheaf}. The unary operations on \(A^X\) are then given by
	\[
		s(\phi)(u,x) \eqdef \phi(\sigma u,x) \qquad\qquad t(\phi)(u,x) \eqdef \phi(\tau u,x)
	\]
	for \(\phi \in A^X\), \(u \in \{1,\sigma,\tau\}\) and \(x \in X\). The evaluation map \(\epsilon_X \from A^X \times X \to A\) takes a couple \((\phi,x) \in A^X \times X\) and sends it to the element \(\phi(1,x) \in A\).
\end{lemma}

\begin{proof}
	Straightforward application of the construction of \cite[Section~I.6, Proposition~1]{MacLane-Moerdijk} for a presheaf category applied on \([\opp{\E}, \Set]\).
\end{proof}

\begin{remark}
	Using the evaluation map---which is the counit of the adjunction \((\placeholder) \times X \adj (\placeholder)^X\)---we can describe the bijection \(\RG(Y \times X , A) \iso \RG(Y , A^X)\) for \(Y \in \RG\): given a morphism \(g \from Y \to A^X\), it is sent to \(f \eqdef \epsilon_X \comp (g \times X) \in \RG(Y \times X , A)\) which maps a couple \((y,x) \in Y \times X\) to the element
	\[
		\epsilon_X\bigl(g(y),x\bigr) = g(y)(1,x) \in A\text{.}
	\]
	The inverse bijection sends	\(f \from Y \times X \to A\) to the morphism \(g \from Y \to A^X\) given by
	\[
		g(y)(u,x) \eqdef u\bigl(g(y)\bigr)(1,x) = g\bigl(u_Y(y)\bigr)(1,x) = f(u_Y(y),x)
	\]
	for \(u\in \{1,\sigma, \tau\}\) and \(x\in X\), where \(u_Y\) (the image of \(u\) through the presheaf associated to \(Y\)) is the unary operation on \(Y\) associated to \(u\). Indeed, the value for \(u=1\) is directly determined by the bijection in the other direction and, for \(u \in \{\sigma,\tau\}\), we use the fact that \(g \from Y \to A^X\) has to be a morphism of reflexive graphs, i.e.\ \(u\circ g=g\circ u_Y\). This observation will be central to endow the reflexive graph \(A^X\) with a group structure---turning it into a precrossed module---and a group action of \(X\).
\end{remark}

Now that we know that \(\RG\) is cartesian closed and that we have a description of its power object, we can apply the construction of \cite[Proposition~5.3]{Gra12} to obtain the right adjoint \(R \from \PXMod = \Grp(\RG) \to \SSES_X(\PXMod)\). We turn the reflexive graph \(A^X\) into a precrossed module by equipping it with the pointwise group structure, which guarantees that the unary operations on \(A^X\) are group morphisms. For instance, we have \(s(\phi\cdot \phi')=s(\phi)\cdot s(\phi')\) since for \((u,x)\in \{1,\sigma, \tau\}\times X\):
\begin{align*}
	s(\phi\cdot\phi')(u,x)=(\phi\cdot\phi')(\sigma u,x)=\phi(\sigma u,x)\cdot\phi'(\sigma u,x)=s(\phi)(u,x)\cdot s(\phi')(u,x)\text{.}
\end{align*}
We define the group-action of \(X\) on \(A^X\) by sending \(\bigl((\phi,x),x'\bigr) \in (A^X \times X) \times X\) to \(\epsilon_X(\phi,x'x)=\phi(1,x'x) \in A\) so, by the adjunction \((\placeholder) \times X \adj (\placeholder)^X\), \(\phi^x(1,x') = \phi(1,x'x)\). As a result, we can then find the value of \(\phi^x\) on an arbitrary \((u,x') \in \{1,\sigma,\tau\} \times X\):
\begin{equation}
	\label{eq - complete description of the action of X on the power object}
	\phi^x(u,x')\overset{(\alpha)}{=}\epsilon_X\left(u_{X\times A^X}(x,\phi), x'\right)\overset{(\beta)}{=}\epsilon_X\left((u_X(x), u(\phi)), x'\right)=u(\phi)(1,x'u_X(x))\overset{(\gamma)}{=}\phi(u,x'u_X(x))
\end{equation}
where \(u_X\) denotes the unary operations associated on \(X\), \((\alpha)\) follows from the adjunction \((\placeholder) \times X \adj (\placeholder)^X\) and the above definition of the action, \((\beta)\) is a consequence from the definition of \(u_{X\times A^X}\) (the unary operations associated on \(X\times A^X\)), and \((\gamma)\) follows from the definition of the unary operation on \(A^X\).

In particular, if we consider \(u\), \(v\in \{1,\sigma,\tau\}\), \(x\), \(x'\in X\) and \(\phi\in A^X\), then
\begin{equation}
	\label{eq - action and composition clear}
	\phi^x(uv,x')=\phi(u\circ v,x'\cdot (v_X\circ u_X)(x))\text{.}
\end{equation}
The reverse order between the two components follows the construction: the composition in \(\{1,\sigma,\tau\}\) (here \(u \circ v\)) is in the reverse order of the one for the unary operations (here \(v_X\circ u_X\)).

The above construction is equivalent to the usual notion of internal action~\cite[Section~4]{BJK05} and in particular can be expressed as the split short exact sequence
\[
	\begin{tikzcd}[cramped]
		0 \arrow[r] & A^X \arrow[r, nmono] & A^X\rtimes X \arrow[r, repi, shift left] & X \arrow[l, mono, shift left] \arrow[r] & 0
	\end{tikzcd}
\]
where \(A^X\rtimes X\) is the usual group-theoretic semi-direct product, i.e.\ the set \(A^X\times X\) equipped the product given by
\[
	(\phi,x)\cdot (\phi',x')=(\phi \cdot \phi'^x, x\cdot x')
\]
for \(\phi\), \(\phi'\in A^X\), \(x\), \(x'\in X\). The unary operations on this semi-direct product \(A^X \rtimes X\) are induced by those of \(A^X\) and \(X\), i.e.\ are given componentwise. Clearly these operations satisfies the unary conditions (see \cref{lemma - RG(V) variety}) so we just have to check that they define group morphisms. Let \(\phi\), \(\phi'\in A^X\) and \(x\), \(x'\in X\). We have
\begin{align*}
	(s\rtimes s_X)\bigl((\phi,x)\cdot (\phi',x')\bigr)   & = (s\rtimes s_X)(\phi \cdot \phi'^x, xx')=\bigl(s(\phi \cdot \phi'^x),s_X(xx')\bigr)                                     \\
	(s\rtimes s_X)(\phi,x)\cdot (s\rtimes s_X)(\phi',x') & = \bigl(s(\phi),s_X(x)\bigr)\cdot \bigl(s(\phi'),s_X(x')\bigr)=\bigl(s(\phi) \cdot s(\phi')^{s_X(x)},s_X(x)s_X(x')\bigr)
\end{align*}
where the second components are equal since \(s_X\) is a morphism of groups. Concerning the first components, we compare their values applied on \((u,y)\in \{1,\sigma,\tau\}\times X\). For the top expression, we have
\[
	\begin{split}
		s(\phi\cdot \phi'^x)(u,y) & =s(\phi)(u,y)\cdot s\bigl(\phi'^x\bigr)(u,y)=\phi(\sigma u,y)\cdot \phi'^x(\sigma u,y) \\
		                          & =\phi(\sigma u,y)\cdot \phi'\bigl(\sigma u, y(u_X\circ s_X)(x)\bigr)
	\end{split}
\]
where the last equation follows from \eqref{eq - action and composition clear}, and, for the bottom expression,
\[
	\begin{split}
		\bigl(s(\phi)\cdot s(\phi')^{s_X(x)}\bigr)(u,y) & = s(\phi)(u ,y)\cdot s(\phi')^{s_X(x)}(u,y) = \phi(\sigma u,y)\cdot s(\phi')\bigl(u,yu_X(s_X(x))\bigr) \\
		                                                & =\phi(\sigma u,y)\cdot \phi'\bigl(\sigma u, y(u_X\circ s_X)(x)\bigr)\text{.}
	\end{split}
\]
Hence, the two expressions are equal and this proves that \(s\rtimes s_X\) is a group morphism. Similar computations show that \(t\rtimes t_X\) is also a morphism in \(\Grp\).

The above construction describes in detail the functor \(R\colon \PXMod\to \SSES_X(\PXMod)\) (the right adjoint to the kernel functor) on the objects. To complete this section, we need to explain how \(R\) acts on morphisms. As observed above, given \(f\in \PXMod(A,B)\), it suffices to give the component of \(R(f)\) between the kernels \(A^X\) and \(B^X\) to build a morphism of split short exact sequences over \(X\) since the inclusions into the middle object are jointly epic. For \(\phi\in A^X\) and \((u,x)\in \{1,\sigma,\tau\}\times X\), we set
\begin{equation}
	\label{eq - R functor on morphisms}
	R(f)(\phi)(u,x)\coloneq f(\phi(u,x))\text{.}
\end{equation}

To conclude, following Gray's construction~\cite{Gra12}, the component of the unit for the adjunction \(K \dashv R\) is determined by its action on the kernel objects
\[
	\eta_A\colon A\to A^X\colon \eta_A(a)(u,x)\coloneq u_A(a)^x
\]
for a given \(A\in \PXMod\), where \(u_A\) is the unary operation on \(A\) associated to \(u\) and the action is the one corresponding to the starting split short exact sequence (see for instance \cite[Section~4]{BJK05}).

\subsubsection{Suitable sections}
\label{sec - suitable sections in PXMod}
In~\cite{DGMVdL}, the authors consider a section \(\theta\colon X\to G\) for the morphism \(f\) viewed in the underlying cartesian closed category where the internal groups are taken (which in their situation is \(\Set\)). This strategy allows to recover the classical Kaluzhnin--Krasner embedding for groups and extend the result from the context of \(\SSES_X(\Grp)\) to the one of \(\SES_X(\Grp)\).

To recover the classical embedding, they define
\[
	\chi\colon \Ker(\CoindArr{f\and 1_X})_{\Grp}\to A \mapping (g,x)\mapsto \theta(x)\cdot g\cdot  \theta(x\cdot f(g))^{-1}
\]
where \(\theta\) is a set-theoretical section of \(f\) and this is where the presentation of \(\Ker(\CoindArr{f\and 1_X})\) (see \cref{subsect - Ker in Grp}) comes into play to check that this assignment indeed defines a group homomorphism. Since the underlying group of this kernel stays the same in \(\PXMod\), we may apply the same procedure and consider
\[
	\chi\colon \Ker(\CoindArr{f\and 1_X})_{\PXMod}\to A \mapping (g,x)\mapsto \theta(x)\cdot g \cdot \theta(x\cdot f(g))^{-1}
\]
for the embedding in \(\PXMod\). We will see in \cref{sec - chi problem with XMod} that this strategy could not be replicated in \(\XMod\) since the underlying group will be a quotient of the kernel here. However, even if we already know that \(\chi\) is a group morphism (thanks to the group presentation which remains valid in the \(\PXMod\) case), we must also check the compatibility with the unary operations to make sure that \(\chi\) is a morphism of precrossed modules as desired.

Without any further assumptions on \(\theta\), this is not the case in general, as the following example shows.

\begin{example}
	\label{ex - chi not compatible}
	Consider the precrossed modules \(X=(\mathbb{Z}/2\mathbb{Z},\mathrm{id},\mathrm{id})\) and \(G=(S_3\times\mathbb{Z}/2\mathbb{Z},p,p)\), where \(S_3\) is the symmetric group on \(\{1,2,3\}\) (written in cycle notation, with permutations composed as functions, i.e.\ from right to left) and \(p\) is the idempotent endomorphism \((a,i)\mapsto(\overline{a},i)\), with \(\overline{a}=e\) for \(a\) even and \(\overline{a}=(1\,2)\) for \(a\) odd, so that \(p\) retracts \(S_3\) onto \(\langle(1\,2)\rangle\). As \(p\) preserves the second component, the projection \(f\colon G\to X \mapping (a,i)\mapsto i\), is a morphism of precrossed modules, with kernel \(A=(S_3\times\{0\},p,p)\).

	Choose the set-theoretical section \(\theta\colon X\to G\) given by \(\theta(0)=(e,0)\) and \(\theta(1)=(\rho,1)\), where \(\rho=(1\,2\,3)\). It is a section of \(f\), but not of the underlying reflexive graph, since \(p(\theta(1))=(e,1)\neq(\rho,1)=\theta(\mathrm{id}(1))\). Write \(s=p\) for the unary operation and take the generator \((g,x)\) with \(g=((1\,2),0)\in A = \Ker f \leq G\) and \(x=1\in X\). As \(g\) is fixed by \(p\) and \(x\) by \(\mathrm{id}\), we have \(s(g,x)=(g,x)\), hence \(\chi(s(g,x))=\chi(g,x)\); and, since \(A\) is normal in \(G\),
	\[
		\chi(g,x)=\theta(1)\cdot g\cdot\theta(1)^{-1}=\bigl(\rho\,(1\,2)\,\rho^{-1},0\bigr)=\bigl((\rho(1)\ \rho(2)),0\bigr)=\bigl((2\,3),0\bigr)\text{,}
	\]
	using that, in the right-to-left convention, conjugation by \(\rho\) relabels the entries of a cycle by \(\rho\).
	This element is not fixed by \(s=p\): indeed \(s(\chi(g,x))=p\bigl((2\,3),0\bigr)=\bigl((1\,2),0\bigr)\), whereas \(\chi(s(g,x))=\bigl((2\,3),0\bigr)\). Thus \(\chi\) does not commute with \(s\) and is not a morphism of precrossed modules.
\end{example}

Therefore, we consider situations where the set-theoretical section \(\theta\) is \emph{compatible with the unary operations}, i.e.\ it is a section of \(f\) in the underlying cartesian closed category \(\RG\) instead of \(\Set\).

\begin{proposition}
	\label{prop:chi morphism}
	If \(f\) admits a section \(\theta \from X \to G\) compatible with the unary operations, then the group morphism \(\chi \colon KL(E)\to A \mapping (g,x) \mapsto \theta(x)\cdot g \cdot \theta(x\cdot f(g))^{-1}\) defines a morphism of precrossed modules.
\end{proposition}

\begin{proof}
	Since this is the same underlying function as in the case of groups, we already know it is a group homomorphism and we only need to check the compatibility with the unary operations.

	Let us do it for the unary operation \(s\), the case for \(t\) being similar. We compute, for \((g,x) \in KL(E)\),
	\[
		\begin{split}
			\chi(s(g,x)) & = \chi(s(g),s(x)) = \theta(s(x)) \cdot s(g) \cdot \theta\bigl(s(x) \cdot f(s(g))\bigr)^{-1} = s(\theta(x)) \cdot s(g) \cdot \theta\bigl(s(x) \cdot s(f(g))\bigr)^{-1} \\
			             & = s(\theta(x)) \cdot s(g) \cdot s\bigl(\theta(x \cdot f(g))^{-1}\bigr) = s\bigl(\theta(x)\cdot g \cdot \theta(x\cdot f(g))^{-1}\bigr) = s(\chi(g,x))\text{,}
		\end{split}
	\]
	finishing the proof.
\end{proof}

Let us explicit examples where the assumptions of the result above are satisfied. A first situation is when \(X\) is a projective reflexive graph since a regular epimorphism in \(\PXMod\) is also a regular epimorphism in \(\RG\).

Another instance where we can find a compatible section is if we impose some conditions on the regular epimorphism \(f\). We will return to the reflexive graph point of view of precrossed modules (see \cref{lemma - RG(V) variety}) and consider a regular epimorphism \(f\colon G\to X\) of precrossed modules (i.e.\ a couple \((f_0\colon G_0\to X_0, f_1\colon G_1\to X_1)\) such that \(f_0\) and \(f_1\) are regular epimorphisms in the underlying category) that is in addition an \emph{internal fully faithful} functor~\cite{Bunge-Pare}, i.e.\ such that the right-hand diagram
\[
	\begin{tikzcd}
		K_1\ar[rr,"k_1", nmono]\arrow[dd, "d^K", shift left=2] \arrow[dd, "c^K"', shift right=2] &  & G_1\ar[rr, "f_1"] \arrow[dd, "d^G", shift left=2] \arrow[dd, "c^G"', shift right=2] &  & X_1 \arrow[dd, "d^X", shift left=2] \arrow[dd, "c^X"', shift right=2] \\
		\\
		K_0 \ar[rr,"k_0" ', nmono] \arrow[uu, "\iota^K" description]                             &  & G_0 \ar[rr, "f_0"'] \arrow[uu, "\iota^G" description]                               &  & X_0\arrow[uu, "\iota^X" description]
	\end{tikzcd}
\]
is a joint pullback (see below). We may notice that the left-hand vertical morphisms are induced by kernel properties and define a reflexive graph.

Since, by assumption, \(f_0\) is a regular epimorphism in \(\Grp\), we can consider a set-theoretical section \(\theta_0\) of it. We claim that by the above joint pullback we can define a set-theoretical section \(\theta_1\) of \(f_1\) such that \((\theta_0,\theta_1)\) defines a morphism in \(\RG\).

First, it is important to notice that we are considering a joint pullback in \(\Grp\). This means a pullback in \(\Grp\) (see the square \((\dagger)\)) that remains a pullback in \(\Set\). This will be the key point to build functions and not group morphisms. Hence we can define \(\theta_1\) as the unique function such that
\[
	f_1 \theta_1=1_{X_1}\text{ and } \IndArr{d^G\and c^G}\theta_1 = (\theta_0\times \theta_0)\IndArr{d^X \and c^X}\text{.}
\]
\[
	\begin{tikzcd}
		X_1\ar[drr, bend left, "1_{X_1}"] \ar[dd, "\IndArr{d^X\and c^X}"'] \ar[dr, induced, "\theta_1"] &                                                                                              &                                   \\
		& G_1\pb[.25]{rd} \ar[dr, phantom, "(\dagger)"] \ar[r, "f_1"] \ar[d, "\IndArr{d^G\and c^G}" '] & X_1\ar[d, "\IndArr{d^X\and c^X}"] \\
		X_0\times X_0\ar[r, "\theta_0\times \theta_0 " ']                                               & G_0\times G_0\ar[r, "f_0\times f_0"']                                                        & X_0\times X_0
	\end{tikzcd}
\]

By construction, the couple \((\theta_0,\theta_1)\) is already compatible with \(d^X\), \(d^G\), \(c^X\) and \(c^G\). It remains to check the compatibility with \(\iota^G\) and \(\iota^X\): it is the case since we verify that
\begin{align*}
	 & \IndArr{d^G \and c^G}e^G \theta_0 = \IndArr{\theta_0\and \theta_0}= \IndArr{d^G \and c^G}\theta_1 e^X \\
	 & f_1 e^G \theta_0 = e^X = f_1 \theta_1 e^X\text{,}
\end{align*}
the conclusion following by uniqueness.

Therefore, we proved:

\begin{lemma}\label{lemma - situation for a suitable section in RG}
	If we consider a fully faithful functor \((f_0,f_1)\) such that \(f_0\) is a surjective group morphism then we can define a morphism of reflexive graphs \((\theta_0,\theta_1)\) such that \(\theta_0\) and \(\theta_1\) are, respectively, set-theoretical sections of \(f_0\) and \(f_1\).\noproof
\end{lemma}

As a consequence, we can apply \cref{prop:chi morphism} to obtain:

\begin{corollary}
	\label{cor - definition of chi}
	If we consider a fully faithful functor \((f_0,f_1)\) such that \(f_0\) is a surjective group morphism, then \(f\) admits a compatible section \(\theta\), and the associated group morphism
	\[
		\chi\colon KL(E)\to A \mapping (g,x)\mapsto \theta(x)\cdot g \cdot \theta(x\cdot f(g))^{-1}
	\]
	is then a morphism of precrossed modules.\noproof
\end{corollary}

\begin{remark}[Reformulation of the previous assumptions]\label{rem - reformulation situation for a suitable section in RG}
	Since the proof of the previous lemma is based on a pullback, we may reformulate the assumptions on \((f_0,f_1)\) as follows: if we consider a regular epimorphism \((f_0,f_1)\) in \(\PXMod\) then it will be a fully faithful functor if and only if the induced arrow \(\IndArr{d^K\and c^K}\colon K_1\to K_0\times K_0\) is an isomorphism. Indeed, since \(k_1\) is the kernel of \(f_1\) and \(k_0\times k_0\) is the kernel of \(f_0\times f_0\) then the square \((\dagger)\) is a pullback if and only if \(\IndArr{d^K\and c^K}\) is an isomorphism (see for instance \cite[Lemma~4.2.4]{BB04}).
\end{remark}

The assumptions of \cref{lemma - situation for a suitable section in RG} are not surprising since they already appear in the context of a Quillen model on the category \(\Cat(\C)\) of internal categories in a given finitely complete category \(\C\):
\begin{proposition}[{\cite[Proposition~5.6]{EKVdL}}]
	If we consider the Grothendieck regular epimorphism topology on a semi-abelian category \(\C\), a functor \(\mathbf{f}\colon \mathbf{G}\to \mathbf{X}\) is a trivial fibration (for the Quillen model on \(\Cat(\C)\) defined in \cite[Theorem~5.5]{EKVdL}) if and only if it is a fully faithful functor such that \(f_0\) is a regular epimorphism.\noproof
\end{proposition}

\subsubsection{Conclusion}\label{subsubsection - conclusion in PXMod}

As a result, if \(f\)---the cokernel morphism in the short exact sequence in \(\PXMod\)---admits a set-theoretical section compatible with the unary operations (see the previous subsection where we discussed several conditions that assure this is the case), then we can construct the embedding
\begin{equation*}
	\label{eq-sum up of KK in PXMod via category}
	\begin{tikzcd}
		\SES_X(\PXMod)\arrow[d, "L"']   & 0 \arrow[r] & A \arrow[r,nmono, "k"] \arrow[d, "\lambda_1"', induced]          & G \arrow[r, "f", repi] \arrow[d, "\iota_1 = \lambda_2"']                             & X \arrow[r] \arrow[d, equal]                                  & 0 \\
		\SSES_X(\PXMod)\arrow[d, "RK"'] & 0 \arrow[r] & KL(E) \arrow[r,nmono] \arrow[d, "\eta_{KL(E)}=\eta_1"', induced] & G+X \arrow[r, "\CoindArr{f \and 1_X}", shift left, repi] \arrow[d, "\phi = \eta_2"'] & X \arrow[r] \arrow[l, "\iota_2", shift left] \arrow[d, equal] & 0 \\
		\SSES_X(\PXMod)\arrow[d, "P"']  & 0 \arrow[r] & KL(E)^X \arrow[r,nmono] \arrow[d, equal]                         & KL(E)^X \rtimes X \arrow[r, shift left,repi] \arrow[d,equal]                         & X \arrow[r] \arrow[d, equal]  \arrow[l, shift left]           & 0 \\
		\SES_X(\PXMod)  \arrow[d]       & 0 \arrow[r] & KL(E)^X \arrow[r,nmono] \arrow[d, induced, "W(\chi)" ']          & KL(E)^X \rtimes X \arrow[r,repi] \arrow[d, "W(\chi)\rtimes 1_X"]                     & X \arrow[r] \arrow[d, equal]                                  & 0 \\
		\SES_X(\PXMod)                  & 0 \arrow[r] & A^X \arrow[r,nmono]                                              & A^X \rtimes X \arrow[r,repi]                                                         & X \arrow[r]                                                   & 0
	\end{tikzcd}
\end{equation*}
where, similarly to the case of groups, we have
\begin{itemize}
	\item \(\lambda_2\) the first inclusion and \(\lambda_1\) its restriction between the kernels;
	\item the couple \((\eta_1,\eta_2)\) is a component of the unit for the adjunction \(K\dashv R\) given by Gray's construction, determined by its value \(\eta_1(a)(u,x) \coloneq x u_{KL(E)}(a) x^{-1}\in KL(E)\), for \(a \in KL(E)\), on the kernel object;
	\item as a consequence of the two previous points, since \(g = \bigl(gf(g)^{-1}\bigr) f(g)\) with \(gf(g)^{-1} \in KL(E)\) and \(f(g) \in X\), we have
	      \[
		      (\phi \circ \iota_1)(g) = \bigl(\eta_1(gf(g)^{-1}),1\bigr) \bigl(1,f(g)\bigr) = (h_g,f(g))
	      \]
	      where \(h_g(u,x) \eqdef x u_{KL(E)}(gf(g)^{-1}) x^{-1} = x u_{KL(E)}(g) \bigl(x u_{KL(E)}(f(g))\bigr)^{-1}\), which corresponds to the element \(\bigl(u_{KL(E)}(g),x\bigr)\) in the presentation of \cref{subsect - Ker in Grp};
	\item Finally, as in the case of groups, by functoriality~\eqref{eq - R functor on morphisms}, we have \(W(\chi)(\psi)=\chi\circ \psi\) for \(\psi \in KL(E)^X\). Therefore, we can conclude that \(G\) embeds in \(A^X\rtimes X\) by
	      \[
		      g \in G \mapsto (W(\chi)\rtimes 1_X)\circ \phi\circ \iota_1(g) = (\chi \comp h_g , f(g)) \in A^X \rtimes X\text{.}
	      \]
\end{itemize}

To sum up, we have the following result:
\begin{theorem}[Kaluzhnin--Krasner in \(\PXMod\)]\label{thm - KK in PXMod}
	Any precrossed module extension
	\[
		\begin{tikzcd}
			0 \arrow[r] & A \arrow[r, "k", nmono] & G \arrow[r, "f", repi] & X \arrow[r] & 0
		\end{tikzcd}
	\]
	where \(f\) admits a reflexive graph section \(\theta \from X \to G\) embeds into the wreath product extension
	\[
		\begin{tikzcd}
			0 \arrow[r] & A^X \arrow[r, nmono] & A^X \rtimes X \arrow[r, repi] & X \arrow[r] & 0 \text{,}
		\end{tikzcd}
	\]
	where the semi-direct product \(A^X\rtimes X\) is associated to the action given in \eqref{eq - complete description of the action of X on the power object}, via the morphism of precrossed module extensions over \(X\) \((\phi_A\colon A\to A^X, \phi_G\colon G\to A^X\rtimes X)\) defined by
	\begin{align*}
		 & \phi_A(a)\eqdef W(\chi)\comp \eta_1(a)\colon (u,x)\in \{1,\sigma,\tau\}\times X\mapsto \theta(x)\cdot u_{KL(E)}(a)\cdot  \theta(x\cdot f(u_{KL(E)}(a)))^{-1}\text{,} \\
		 & \phi_G(g)\eqdef(\chi \comp h_g , f(g))
	\end{align*}
	where
	\[
		\chi \comp h_g(u,x)=\theta(x)\cdot u_{KL(E)}(g)\cdot  \theta(x\cdot f(u_{KL(E)}(g)))^{-1}\text{.}
	\]
\end{theorem}
\begin{proof}
	We already explained the construction of this morphism before so we only have to check that it is a monomorphism to obtain the desired embedding.

	Since \(\PXMod\) is semi-abelian---being the category of internal groups in \(\RG\)---it is in particular protomodular, so pulling back reflects monomorphisms~\cite[Lemma~3.1.20]{BB04}. As the right-hand component of the morphism of extensions is the identity \(1_X\), the left square of this morphism is indeed a pullback~\cite[Lemma~4.2.4]{BB04} so it is enough to check that the first component \(\phi_A\colon A\to A^X\) is a monomorphism; the middle component \(\phi_G\) is then automatically monic as well.

	By construction, \(\phi_A=W(\chi)\comp \eta_1\comp \lambda_1\) is the composite displayed in the left-hand column of the diagram above; we check that it is injective. Let \(a\in A\). The morphism \(\lambda_1\) is the restriction to the kernels of the coproduct inclusion \(\lambda_2=\iota_1\colon G\to G+X\), so \(\lambda_1(a)\) is the element \((a,e)\) in the presentation of \cref{subsect - Ker in Grp}---recall that \(f(a)=e\) since \(a\in A=\Ker(f)\). Applying the unit \(\eta_1\), we obtain
	\[
		\eta_1\bigl(\lambda_1(a)\bigr)(u,x)=u_{KL(E)}\bigl(\lambda_1(a)\bigr)^x=x\cdot u_{KL(E)}\bigl(\lambda_1(a)\bigr)\cdot x^{-1}\in KL(E)
	\]
	for \((u,x)\in \{1,\sigma,\tau\}\times X\). Since \(\lambda_1\colon A\to KL(E)\) is a morphism of precrossed modules, it commutes with the unary operations, so \(u_{KL(E)}\bigl(\lambda_1(a)\bigr)=\lambda_1\bigl(u_A(a)\bigr)\). Concretely, as \(A=\Ker(f)\) is a kernel in \(\PXMod\), it is closed under the unary operations \(s_G\) and \(t_G\) (see \cref{thm - Higgins kernel distributive omega groups}), so \(u_G(a)=u_A(a)\) still belongs to \(A\); in particular \(f(u_A(a))=e\), and under the presentation \(\lambda_1(a)=(a,e)\) maps to \(\lambda_1(u_A(a))=(u_A(a),e)\), so the element above is \((u_A(a),x)\) in the presentation. Postcomposing with \(\chi\) (see \cref{prop:chi morphism}) finally gives
	\[
		\phi_A(a)(u,x)=\chi\bigl(u_A(a),x\bigr)=\theta(x)\cdot u_A(a)\cdot \theta(x)^{-1}\in A\text{.}
	\]
	Let us assume that the above expression is equal to the neutral element \(e_A\) for all \((u,x)\in \{1,\sigma,\tau\}\times X\), i.e.\ that \(\phi_A(a)\) is equal to the neutral element in \(A^X\). In particular, its component at \(u=1\) is the assignment \(a\mapsto \bigl(x\mapsto \theta(x)\, a\, \theta(x)^{-1}\bigr)\), which is precisely the first component of the classical Kaluzhnin--Krasner embedding in \(\Grp\). Evaluating it at \(x=e\) yields \(\theta(e)\, a\, \theta(e)^{-1}\); since \(A\) is normal in \(G\), conjugation by the fixed element \(\theta(e)\) is injective on \(A\). Hence \(\phi_A\) is a monomorphism, which concludes the proof.
\end{proof}

\section{Discussion about crossed modules}
The strategy used in the above section to express the embedding in \(\PXMod\) was to adapt all the steps made for the classical embedding in \(\Grp\). As already recalled before, we need a \LACC{} category to define the right adjoint \(R\). The strategy for \(\PXMod\) was the same as for \(\Grp\) since they are both \LACC{} for the same reason: they are the category of internal groups over a cartesian closed category with pullbacks.

In this section, we will consider \(\XMod\) which is also \LACC{} for the same argument: it is the category of internal groups in a cartesian closed category, \(\Cat\). Hence, we could try to follow the same procedure. However, we will face some major issues. Since the description of the power object of \(\Cat\) is classical and the construction of \(R\) follows directly from Gray's construction, we will only focus on the new results: given a regular epimorphism \(f\colon G\to X\) in \(\XMod\) with kernel \(A\), we will describe \(\Ker(\CoindArr{f\and 1_X})\) in \(\XMod\) and try to construct a morphism \(\chi\colon \Ker(\CoindArr{f\and 1_X})_{\XMod}\to A\).

To express those new results, we need to recall some basics results about the adjunction between \(\XMod\) and \(\PXMod\).

\subsection{Commutator theory and the adjunction between \(\PXMod\) and \(\XMod\)}\label{subsection - general results about XMod}

\subsubsection{The conditions \SH{} and \NH{}}
It is well known that, for a semi-abelian category \(\C\), \(\XMod(\C)\) is a full subcategory of \(\PXMod(\C)\). The description of the reflector relies on the following characterisation of an internal category over a finitely complete Mal'tsev category, and so, in particular, over a semi-abelian category: a reflexive graph \((G_0,G_1, d,c,\iota)\) admits an internal category structure if and only if \([\Eq(d), \Eq(c)]=\Delta_{G_1}\), where \(\Eq(d)\) and \(\Eq(c)\) are the respective kernel pairs of \(d\) and \(c\) and \(\Delta_{G_1}\) is the smallest equivalence relation on \(G_1\) (see for instance~\cite{BB04}). The commutator used here is called \emph{Smith's commutator}~\cite{S76,P95} and it is defined on equivalence relations.

When \(\C\) satisfy the ``\emph{Smith is Huq}'' (in short \SH{})~\cite{MVdL12} and ``\emph{normality of Higgins}'' (in short \NH{})~\cite{AlanThesis} conditions, such as the category of groups\footnote{It is important to point out that the two conditions \NH{} and \SH{} are independent of each other (see \cite[Section~5.3]{CGrayVdL1} for some examples).} (see~\cite{MLan98} for \SH{} and \cite{AlanThesis} for \NH{})\@, then a reflexive graph \((G_0,G_1, d,c,\iota)\) is an internal category if and only if \([\Ker(d),\Ker(c)]=0\), where \([\placeholder,\placeholder]\) is the Higgins commutator~\cite{MM10b}, i.e.\ the regular image of the composite
\[
	\begin{tikzcd}[column sep=huge]
		\Ker(d)\cosmash \Ker(c) \arrow[r, "\kappa_{\Ker(d),\Ker(c)}", nmono] & \Ker(d)+ \Ker(c)\arrow[r, "\CoindArr{\ker(d)\and \ker(c)}"] & G_1
	\end{tikzcd}
\]
where \((\Ker(d)\cosmash \Ker(c),\kappa_{\Ker(d),\Ker(c)})\) denotes the kernel of the canonical map from \(\Ker(d)+ \Ker(c)\) to \(\Ker(d)\times\Ker(c)\). In \(\Grp\), this latter commutator is the usual commutator of groups. Furthermore, we recall that, under \SH{} and \NH{}\@, since \(\Ker(d)\) and \(\Ker(c)\) are normal subobjects of \(G_1\), then so is \([\Ker(d),\Ker(c)]\). As a consequence, we can quotient with it in the following lemma:

\begin{lemma}
	\label{lemma:adjunction between XMod PXMod under SH with Huq}
	In a semi-abelian category \(\C\) satisfying \SH{}\@ and \NH{}\@, the forgetful functor \(\XMod(\C)\to \PXMod(\C)\) admits a left adjoint.
\end{lemma}
\begin{proof}
	Using the equivalences of category, we will prove that the forgetful functor \(U\colon \Grpd(\C)\to \RG(\C)\) admits a left adjoint \(L\). Let \((G_0,G_1,d,c,\iota)\) a reflexive graph and consider the diagram
	\begin{equation}
		\label{diag-adjunction SH non variety}
		\begin{tikzcd}
			G_1 \arrow[r, "d", shift left=2] \arrow[r, "c"', shift right=2] \arrow[d, "q"', repi]                           & G_0 \arrow[l, "\iota" description] \arrow[d, equal] \\
			\frac{G_1}{[\Ker(d),\Ker(c)]} \arrow[r, "\overline{d}", shift left=2] \arrow[r, "\overline{c}"', shift right=2] & G_0 \arrow[l, "q \iota" description]
		\end{tikzcd}
	\end{equation}
	where \(\overline{d}\) is such that \(d=\overline{d}q\) and, similarly, \(\overline{c}\) is such that \(c=\overline{c}q\). Indeed, we can prove that since \(\Ker(d)\), \(\Ker(c)\) are normal in \(G_1\), then \([\Ker(d),\Ker(c)]\leq \Ker(d)\) and \([\Ker(d),\Ker(c)]\leq \Ker(c)\). This diagram defines a morphism of reflexive graphs since its bottom part is a reflexive graph. In addition, this bottom reflexive graph is an internal groupoid since \({[\Ker(\overline{d}), \Ker(\overline{c})]=0}\).

	Moreover, we can see that the above vertical morphism \((1_{G_0},q)\) defines the \((G_0,G_1,d,c,\iota)\)-component of the unit of the adjunction.

	As a consequence, we can easily prove the adjunction: let \((g_0,g_1)\colon (G_0,G_1,d,c,\iota)\to (X_0,X_1,d^X,c^X,\iota^X)\) be a morphism of reflexive graphs with an internal groupoid as codomain. Since \((g_0,g_1)\) determines a morphism of reflexive graphs, there exists unique morphisms \(\overline{g_1}^d\) and \(\overline{g_1}^c\) such that
	\[g_1\ker(d)=\ker(d^X)\overline{g_1}^d\quad \text{and} \quad g_1\ker(c)=\ker(c^X)\overline{g_1}^c.\]

	If we can prove that \(g_1\ker(q)=0\) then the proof is complete since the universal morphism of internal groupoid will be \((g_0,\overline{g_1})\) where \(\overline{g_1}\) is the unique morphism such that \(g_1=\overline{g_1}q\).
	However, if \(e\) denotes the regular epimorphism part of the image factorisation of \(\CoindArr{\ker(d)\and \ker(c)}\kappa_{\Ker(d),\Ker(c)}\) then it is clear that it suffice to prove that \(g_1\ker(q)e=0\):
	\begin{align*}
		g_1\ker(q)e & = g_1\CoindArr{\ker(d)\and \ker(c)}\kappa_{\Ker(d),\Ker(c)}= \CoindArr{\ker(d^X)\and \ker(c^X)}(\overline{g_1}^d + \overline{g_1}^c)\kappa_{\Ker(d),\Ker(c)} \\
		            & = \CoindArr{\ker(d^X)\and \ker(c^X)}\kappa_{\Ker(d^X),\Ker(c^X)}(\overline{g_1}^d \cosmash \overline{g_1}^c)
	\end{align*}
	where, since \((X_0,X_1,d^X,c^X,\iota^X)\) is an internal groupoid, the commutator \([\Ker(d^X),\Ker(c^X)]\) is trivial, i.e.\ \(\CoindArr{\ker(d^X)\and \ker(c^X)}\kappa_{\Ker(d^X),\Ker(c^X)}=0\).
\end{proof}

\begin{remark}
	\label{rem:ajunction with s and t}
	In the context of a semi-abelian variety \(\V\) satisfying \SH{} and \NH{}\@, by \cref{lemma - RG(V) variety} the previous diagram~\eqref{diag-adjunction SH non variety} may be reformulated as
	\[
		\begin{tikzcd}
			X \arrow[r, "s", shift left] \arrow[r, "t"', shift right] \arrow[d, "q"', repi]                           & X \arrow[d, "q", repi]      \\
			\frac{X}{[\Ker(s),\Ker(t)]} \arrow[r, "\overline{s}", shift left] \arrow[r, "\overline{t}"', shift right] & \frac{X}{[\Ker(s),\Ker(t)]}
		\end{tikzcd}
	\]
	where \(\overline{s}\eqdef q\iota \overline{d}\) and \(\overline{t}\eqdef q\iota\overline{c}\), which implies in particular that all the squares commute.
\end{remark}

\begin{corollary}
	\label{cor - reflexive graph with SH and NH}
	If \(\V\) is a semi-abelian variety satisfying \SH{}\@ and \NH{}\@, a precrossed module \((G,s,t)\) is a crossed module if and only if \([\Ker(s),\Ker(t)]=0\).\noproof
\end{corollary}

If we combine the previous lemma with \cref{cor: coproducts in PXMod} for a given semi-abelian variety \(\V\) satisfying \SH{} and \NH{}\@, then \(\XMod(\V)\) admits binary coproducts:
\begin{lemma}
	\label{lemma: binary coproduct in XMod}
	Let \(\V\) a semi-abelian variety satisfying \SH{} and \NH{}\@. Consider two crossed modules \((G_1,s_1,t_1)\) and \((G_2,s_2,t_2)\). Their binary coproduct in \(\XMod(\V)\) exists and is given by
	\[
		\Biggl(\frac{G_1+_{\V} G_2}{[\Ker(s_1+_{\V} s_2) , \Ker(t_1+_{\V} t_2)]} \:,\: \overline{s_1+_{\V} s_2} \:,\: \overline{t_1+_{\V} t_2}\Biggr)
	\]
	where the unary operations are described in the remark below and the inclusions are the compositions of the inclusion in \(\V\) followed by the quotient \(q\).\noproof
\end{lemma}

If the underlying semi-abelian variety of algebras \(\V\) is clear, then we will omit the index in the construction of the binary coproducts in \(\XMod(\V)\).

\begin{remark}
	\label{rem:induced morphism for coproducts for XMod in Grp}
	The unary operations \(\overline{s_1+ s_2}\) and \(\overline{t_1+ t_2}\) are the unique morphisms such that all the squares of
	\[
		\begin{tikzcd}
			G_1+_{\V} G_2 \arrow[r, "s_1+ s_2", shift left] \arrow[r, "t_1+ t_2"', shift right] \arrow[d, "q"', repi]                                      & G_1+ G_2 \arrow[d, "q", repi]                      \\
			\frac{G_1+ G_2}{[\Ker(s_1+ s_2) , \Ker(t_1+ t_2)]} \arrow[r, "\overline{s_1+ s_2}", shift left] \arrow[r, "\overline{t_1+ t_2}"', shift right] & \frac{G_1+ G_2}{[\Ker(s_1+ s_2) , \Ker(t_1+ t_2)]}
		\end{tikzcd}
	\]
	commute.

	For instance in the category of groups, for \(g_1\in G_1\) and \(g_2\in G_2\), we have
	\[
		\begin{split}
			(\overline{s_1+ s_2})((g_1g_2)[\Ker(s_1+ s_2) , \Ker(t_1+ t_2)]) & = (s_1(g_1)s_2(g_2))[\Ker(s_1+ s_2) , \Ker(t_1+ t_2)]  \\
			(\overline{t_1+ t_2})((g_1g_2)[\Ker(s_1+ s_2) , \Ker(t_1+ t_2)]) & = (t_1(g_1)t_2(g_2))[\Ker(s_1+ s_2) , \Ker(t_1+ t_2)].
		\end{split}
	\]
\end{remark}

\begin{remark}[Reformulation of the binary coproducts in \(\XMod(\V)\)]
	\label{rem - reformulation binary coproducts in XMod with classical def}
	As for \(\PXMod\), since the construction of the Kaluzhnin--Krasner embedding requires the construction of the right adjoint \(R\colon \XMod\to \SSES_X(\XMod)\), we need to consider the power object in \(\Cat\) (the category of small categories) to follow Gray's construction. Moreover, as done in \cref{subsubsec - power object in PXMod}, we endow it with a group structure and an action.

	We reformulate \cref{lemma: binary coproduct in XMod}: given two crossed modules \((X_0,X_1, d^X, c^X, \iota^X)\) and \((G_0,G_1,d^G,c^G,\iota^G)\) in a semi-abelian variety \(\V\) satisfying \SH{} and \NH{}\@, their binary coproduct is given by the bottom reflexive graph in the diagram
	\[
		\begin{tikzcd}[column sep=large]
			G_1+X_1 \arrow[rr, "d^G+d^X", shift left=2] \arrow[rr, "c^G+c^X"', shift right=2] \arrow[d, "q"', repi]                                       &  & G_0+X_0 \arrow[ll, "\iota^G+\iota^X" description] \arrow[d, equal] \\
			\frac{G_1+X_1}{[\Ker(d^G+d^X),\Ker(c^G+c^X)]} \arrow[rr, "\overline{d^G+d^X}", shift left=2] \arrow[rr, "\overline{c^G+c^X}"', shift right=2] &  & G_0+X_0 \arrow[ll, "q (\iota^G+\iota^X)" description]
		\end{tikzcd}.
	\]
\end{remark}

\subsubsection{The category of groups satisfies \SH{} and \NH{}\@: applications}
We are now focusing on the category \(\Grp\) which is a semi-abelian variety that satisfies \SH{} and \NH{}\@.

The first important results concern the understanding of \(\Ker(s)\) and \(\Ker(t)\). Indeed, as pointed out in \cref{cor - reflexive graph with SH and NH}, those kernels are used to detect whether a given precrossed module is a crossed module or not.

We explicitly describe \(\Ker(s)\) and \(\Ker(t)\) in \(\Grp\): they are, respectively, the subgroups of \(G_1\)
\[
	\{xs(x)^{-1} \mid x\in G_1\} \qquad\text{and}\qquad \{yt(y)^{-1} \mid y\in G_1\}\text{.}
\]
Note that these sets are the subgroups themselves, not generating sets.

Since the proof is symmetric, we will only discuss \(\Ker(s)\). Let \(x\in G_1\). If \(x\in \Ker(s)\), then \(xs(x)^{-1}=x\) and otherwise we have
\[
	s(xs(x)^{-1}) = s(x) s(s(x)^{-1})=s(x)s(x)^{-1}=e
\]
where we use \cref{rem:s and t idempotent}. Hence it is clear that \(\{xs(x)^{-1} \mid x\in G_1\}\) is the kernel of \(s\).

\begin{remark}[Induced morphisms by the binary coproducts in \(\XMod\)]
	\label{remark: induced morphism for the binary coproduct in XMod}
	In \cref{subsubsection KLE in XMod}, we will describe the kernel of a morphism out of the binary coproduct as done for \(\PXMod\) with \cref{lemma: kernel in PXMod}. Therefore, we will explain how to understand its construction in \(\XMod\) based on \cref{lemma: binary coproduct in XMod}.

	Let \(g_1\colon (G_1,s_1,t_1)\to (X,s_X,t_X)\) and \(g_2\colon (G_2,s_2,t_2)\to (X,s_X,t_X)\) be two crossed module morphisms, and consider the following diagram in \(\Grp\).
	\[
		\begin{tikzcd}
			& G_1+G_2\arrow[rrr, "q", repi]\arrow[dd, "\CoindArr{g_1\and g_2}", induced] &                                            &                                             & {\frac{G_1+G_2}{[\Ker(s_1+s_2), \Ker(t_1+t_2)]}}\arrow[dd,"\overline{\CoindArr{g_1\and g_2}}", induced]                                               \\
			G_1\arrow[ru, "\iota_1"]\arrow[dr, "g_1"'] &                                                                            & G_2\arrow[lu, "\iota_2"']\arrow[ld, "g_2"] & G_1\arrow[ru, "q\iota_1"]\arrow[dr, "g_1"'] &                                                                                                         & G_2\arrow[lu, "q\iota_2"']\arrow[ld, "g_2"] \\
			& X\arrow[rrr, equal]                                                        &                                            &                                             & X
		\end{tikzcd}
	\]

	Let us recall that since \(g_1\) and \(g_2\) are morphisms of (pre)crossed modules (i.e.\ compatible with the unary operations) then so is \(\CoindArr{g_1\and g_2}\). This implies in particular that \(\CoindArr{g_1\and g_2}\bigl([\Ker(s_1+s_2), \Ker(t_1+t_2)]\bigr)\in [\Ker(s_X),\Ker(t_X)]=\{e_X\}\). Indeed, given \(a\), \(b\in G_1\) and \(a'\), \(b'\in G_2\) then
	\[
		\begin{split}
			\CoindArr{g_1\and g_2} & \bigl((aa'(s_1(a)s_2(a'))^{-1})(bb'(t_1(b)t_2(b'))^{-1})(s_1(a)s_2(a'))(aa')^{-1}(t_1(b)t_2(b'))(bb')^{-1}\bigr) \\
			                       & = (g_1(a)g_2(a')(g_1(s_1(a))g_2(s_2(a')))^{-1})(g_1(b)g_2(b')(g_1(t_1(b))g_2(t_2(b')))^{-1}                      \\
			                       & \qquad \cdot (g_1(s_1(a))g_2(s_2(a')))(g_1(a)g_2(a'))^{-1}(g_1(t_1(b))g_2(t_2(b')))(g_1(b)g_2(b'))^{-1}          \\
			                       & = (g_1(a)g_2(a')(s_X(g_1(a))s_X(g_2(a')))^{-1})(g_1(b)g_2(b')(t_X(g_1(b))t_X(g_2(b')))^{-1}                      \\
			                       & \qquad \cdot (s_X(g_1(a))s_X(g_2(a')))(g_1(a)g_2(a'))^{-1}(t_X(g_1(b))t_X(g_2(b')))(g_1(b)g_2(b'))^{-1}          \\
			                       & = xs_X(x)^{-1}y t_X(y)^{-1}s_X(x)x^{-1}s_X(y)y^{-1}\in [\Ker(s_X),\Ker(t_X)]
		\end{split}
	\]
	where \(x=g_1(a)g_2(a')\in X\) and \(y=g_1(b)g_2(b')\in X\).

	As a result, we have an induced morphism \(\overline{\CoindArr{g_1\and g_2}}\) in \(\Grp\) such that \(\CoindArr{g_1\and g_2}=\overline{\CoindArr{g_1\and g_2}}q\). Therefore, this induced morphism is compatible with the unary operations since so is \(\CoindArr{g_1\and g_2}\).

	Finally, it is not surprising that the above construction defines the induced morphism in \(\XMod\): \(\overline{\CoindArr{g_1\and g_2}}\) is the image of \(\CoindArr{g_1\and g_2}\) through the reflector described in \cref{rem:ajunction with s and t}. Moreover, \(\overline{\CoindArr{g_1\and g_2}}\) is the universal morphism induced by the binary coproduct in \(\XMod\).

	In addition, this construction induces a morphism of short exact sequences as below.
	\[
		\begin{tikzcd}
			{{[\Ker(s_1+s_2), \Ker(t_1+t_2)]}}\arrow[r, nmono] \arrow[d, induced] & G_1+G_2 \arrow[r, "q", repi]\arrow[d, "\CoindArr{g_1\and g_2}"] & {\frac{G_1+G_2}{[\Ker(s_1+s_2), \Ker(t_1+t_2)]}}\arrow[d, "\overline{\CoindArr{g_1\and g_2}}"] \\
			\{e_X\}\arrow[r, nmono]                                               & X\arrow[r, equal]                                               & X
		\end{tikzcd}
	\]
\end{remark}

\subsection{The kernel \texorpdfstring{\(KL(E)\)}{KL(E)} in crossed modules}\label{subsubsection KLE in XMod}
After describing the kernel in \(\PXMod\) (see \cref{lemma: kernel in PXMod}), let us consider its subvariety \(\XMod\). Let \((G,s_G,t_G)\) and \((X,s_X,t_X)\) be two crossed modules. Unlike the description of the binary coproducts in \(\XMod\) (see \cref{lemma: binary coproduct in XMod}), the group underlying the kernel in \(\XMod\) is now a quotient of the one in \(\Grp\): it is the quotient of \(\Ker(\CoindArr{f\and 1_X})_{\Grp}\) by the commutator \([\Ker(s_G+s_X),\Ker(t_G+t_X)]\). Let us first show that this commutator indeed lies in \(\Ker(\CoindArr{f\and 1_X})_{\Grp}\). Consider the element
\begin{equation}\label{eq - a word in the commutator lies in the kernel}
	gs_X(f(g))^{-1}ht_X(f(h))^{-1}s_G(g)f(g)^{-1}t_G(h)f(h)^{-1}\in [\Ker(s_G+s_X),\Ker(t_G+t_X)]
\end{equation}
where \(g\), \(h\in G\). We can decompose it as a concatenation of words belonging to \(\Ker(\CoindArr{f\and 1_X})_{\Grp}\). More precisely, we can express it via the generators
\[
	xg'f(g')^{-1}x^{-1}
\]
for \(x\in X\) and \(g'\in G\).

It is enough to observe that
\[
	yf(h)h^{-1}y^{-1}=\bigl(yhf(h)^{-1}y^{-1}\bigr)^{-1}
\]
where \(y\in X\) and \(h\in G\). Therefore, we have our claim by observing that for \(g\in G\) and \(x\in X\), we have
\begin{align*}
	gx & =\bigl(gf(g)^{-1}\bigr) f(g)x       &  & \text{(\(X\) component on the right and a generator on the left)}  \\
	   & =f(g)x\bigl(x^{-1}f(g)^{-1}gx\bigr) &  & \text{(\(X\) component on the left and a generator on the right)}  \\
	xg & =xf(g)\bigl(f(g)^{-1}g\bigr)        &  & \text{(\(X\) component on the left and a generator on the right)}  \\
	   & =\bigl(xgf(g)^{-1}x^{-1}\bigr)xf(g) &  & \text{(\(X\) component on the right and a generator on the left).}
\end{align*}

Therefore the element~\eqref{eq - a word in the commutator lies in the kernel} can be written as
\[
	\begin{split}
		gs_X(f(g))^{-1}ht_X(f(h))^{-1}s_G(g)f(g)^{-1}t_G(h)f(h)^{-1} & = \omega_1 x_1ht_X(f(h))^{-1}s_G(g)f(g)^{-1}t_G(h)f(h)^{-1}              \\
		                                                             & =\omega_1 \omega_2 x_2 s_G(g)f(g)^{-1}t_G(h)f(h)^{-1}                    \\
		                                                             & =\omega_1 \omega_2 \omega_3 x_3 t_G(h)f(h)^{-1}                          \\
		                                                             & =\omega_1 \omega_2\omega_3 \omega_4 x_3 f(t_G(h))f(h)^{-1}               \\
		                                                             & = \omega_1\omega_2 \omega_3 \omega_4 \in \Ker(\CoindArr{f\and 1_X})_\Grp
	\end{split}
\]
where
\begin{itemize}
	\item \(\omega_1=gf(g)^{-1}\in \Ker(\CoindArr{f\and 1_X})_\Grp\) and \(x_1=f(g)s_X(f(g))^{-1}\in X\);
	\item \(\omega_2=x_1hf(h)^{-1}x_1^{-1} \in \Ker(\CoindArr{f\and 1_X})_\Grp\) and \(x_2=x_1f(h)t_X(f(h))^{-1}\in X\);
	\item \(\omega_3=x_2s_G(g)f(s_G(g))^{-1}x_2^{-1}\in \Ker(\CoindArr{f\and 1_X})_\Grp\) and \(x_3=x_2f(s_G(g))f(g)^{-1}\in X\);
	\item \(\omega_4=x_3t_G(h)f(t_G(h))^{-1}x_3^{-1}\in \Ker(\CoindArr{f\and 1_X})_\Grp\) and
	      \[
		      \begin{split}
			      x_3 f(t_G(h))f(h)^{-1} & = x_2 s_X(f(g))f(g)^{-1}t_X(f(h))f(h)^{-1}                                 \\
			                             & = x_1f(h)t_X(f(h))^{-1}s_X(f(g))f(g)^{-1}t_X(f(h))f(h)^{-1}                \\
			                             & = f(g)s_X(f(g))^{-1}f(h)t_X(f(h))^{-1}s_X(f(g))f(g)^{-1}t_X(f(h))f(h)^{-1}
		      \end{split}
	      \]
	      hence, \(x_3 f(t_G(h))f(h)^{-1}\in [\Ker(s_X),\Ker(t_X)]=\{ e_X\}\) since \((X,s_X,t_X)\) is a crossed module.
\end{itemize}

The previous observations lead to the description of the kernel of the induced morphism \(\CoindArr{f\and 1_X}\) in \(\XMod\):

\begin{theorem}
	\label{thm - kernel in XMod final version}
	Let \((G,s_G,t_G)\) and \((X,s_X,t_X)\) be two crossed modules, and consider their binary coproduct
	\[
		\Biggl(G+_{\XMod}X \:,\: \overline{s_G+ s_X} \:,\: \overline{t_G+ t_X}\Biggr)\text{.}
	\]
	For a morphism \(f\colon G\to X\) of crossed modules, the kernel of \(\overline{\CoindArr{f\and 1_X}}\) (see \cref{rem:induced morphism for coproducts for XMod in Grp} for the notation) in \(\XMod\) is the crossed module
	\[
		\Biggl(\frac{\Ker(\CoindArr{f\and 1_X})_\Grp}{[\Ker(s_G+ s_X) , \Ker(t_G+ t_X)]} \:,\: s_\XMod \:,\: t_\XMod\Biggr)
	\]
	where \(\Ker(\CoindArr{f\and 1_X})_\Grp\) denotes the kernel in groups and where \(s_\XMod\) and \(t_\XMod\) denote, respectively, the quotient morphisms induced by \(\overline{s_G+ s_X}\) and \(\overline{t_G+ t_X}\) to the kernel \(\frac{\Ker(\CoindArr{f\and 1_X})_\Grp}{[\Ker(s_G+ s_X) , \Ker(t_G+ t_X)]}\).
\end{theorem}

The proof of this theorem is split into the following lemmas. Therefore, we will use the same assumptions and the same notations until the end of the section.

\begin{lemma}
	\label{lemma - commutator in a normal subgroup of Ker}
	The commutator \([\Ker(s_G+s_X), \Ker(t_G+t_X)]\) is a normal subgroup of \(\Ker(\CoindArr{f\and 1_X})\).
\end{lemma}

\begin{proof}
	First we can easily see that it is a subgroup. This follows from the fact that \(\CoindArr{f\and 1_X}\) factorises through \(G+_{\XMod}X\) and that \([\Ker(s_G+s_X), \Ker(t_G+t_X)]\) is the kernel of the factorisation arrow \(G+X \to G+_{\XMod}X\).
	\[
		\begin{tikzcd}
			{[\Ker(s_G+s_X), \Ker(t_G+t_X)]} \arrow[rd, nmono] \arrow[d, induced]                                                                                                         \\
			\Ker(\CoindArr{f\and 1_X})_{\Grp} \arrow[r, nmono] & G+X \arrow[rr, "\CoindArr{f\and 1_X}"] \arrow[rd, repi] &                                                            & X \\
			&                                                         & G+_{\XMod}X \arrow[ru, "\overline{\CoindArr{f\and 1_X}}"']
		\end{tikzcd}
	\]
	where \(G+_{\XMod}X\coloneq\frac{G+ X}{[\Ker(s_G+ s_X) , \Ker(t_G+ t_X)]}\).

	Moreover, since \([\Ker(s_G+s_X), \Ker(t_G+t_X)]\) is normal in \(G+X\), the inclusion above tells us that \([\Ker(s_G+s_X), \Ker(t_G+t_X)]\) is a normal subgroup of \(\Ker(\CoindArr{f\and 1_X})\).
\end{proof}

\begin{lemma}
	\label{lem:kernel XMod}
	The crossed module
	\[
		\Biggl(\frac{\Ker(\CoindArr{f\and 1_X})_\Grp}{[\Ker(s_G+ s_X) , \Ker(t_G+ t_X)]} \:,\: s_\XMod \:,\: t_\XMod\Biggr)
	\]
	is the kernel of \(\overline{\CoindArr{f\and 1_X}}\).
\end{lemma}
\begin{proof}
	By some isomorphism theorems, we have
	\[
		\frac{(G+X)/[\Ker(s_G+s_X) , \Ker(t_G+t_X)]}{\Ker(\CoindArr{f\and 1_X})_{\Grp}/[\Ker(s_G+s_X) , \Ker(t_G+t_X)]} \iso \frac{G+X}{\Ker(\CoindArr{f\and 1_X})_{\Grp}} \iso X \iso \frac{G+_{\XMod}X}{\Ker\overline{\CoindArr{f \and 1_X}}}
	\]
	and thus our kernel is \(\Ker\overline{\CoindArr{f \and 1_X}} \iso \Ker(\CoindArr{f\and 1_X})_{\Grp}/[\Ker(s_G+s_X) , \Ker(t_G+t_X)]\) (since \(G+_{\XMod}X\) is the quotient \((G+X)/[\Ker(s_G+s_X) , \Ker(t_G+t_X)]\)).
\end{proof}

\begin{remark}\label{rem - alternative proof for the kernel}
	Another possible way of understanding the above chain of isomorphisms is expressed in
	\[
		\adjustbox{scale=0.65}{\begin{tikzcd}
				& 0 \arrow[d]                                                                                                & 0 \arrow[d]                                                              & 0\arrow[d]                                                     \\
				{[\Ker(s_{\Ker(\CoindArr{f\and 1_X})_\Grp}), \Ker(t_{\Ker(\CoindArr{f\and 1_X})_\Grp})]} \arrow[d, nmono] \arrow[r, "\iota", mono, induced]             & {[\Ker(s_G+s_X), \Ker(t_G+t_X)]} \arrow[d, nmono]\arrow[r, equal]                                          & {[\Ker(s_G+s_X), \Ker(t_G+t_X)]} \arrow[d, nmono]\arrow[r]               & {[\Ker(s_X),\Ker(t_X)]}=\{e_X\} \arrow[d, nmono] \arrow[r] & 0 \\
				\Ker(\CoindArr{f\and 1_X})_\Grp \arrow[r, equal]\arrow[d, repi]                                                                                         & \Ker(\CoindArr{f\and 1_X})_\Grp \arrow[r, nmono]\arrow[d, repi]                                            & G+X\arrow[r, "\CoindArr{f\and 1_X}", repi] \arrow[d, repi]               & X\arrow[d, repi] \arrow[r]                                 & 0 \\
				\frac{\Ker(\CoindArr{f\and 1_X})_\Grp}{[\Ker(s_{\Ker(\CoindArr{f\and 1_X})_\Grp}), \Ker(t_{\Ker(\CoindArr{f\and 1_X})_\Grp})]} \arrow[r, induced, repi] & \frac{\Ker(\CoindArr{f\and 1_X})_\Grp}{[\Ker(s_G+s_X), \Ker(t_G+t_X)]} \arrow[r, induced, nmono] \arrow[d] & G+_{\XMod}X\arrow[r, "\overline{\CoindArr{f\and 1_X}}"', repi] \arrow[d] & \frac{X}{[\Ker(s_X),\Ker(t_X)]}\cong X \arrow[r] \arrow[d] & 0 \\
				& 0                                                                                                          & 0                                                                        & 0
			\end{tikzcd}}
	\]
	where we can apply the \(3 \times 3\) lemma~\cite[Theorem~4.2.7]{BB04} on the right-hand part of the diagram.

	A third alternative proof would be to apply the Snake lemma~\cite[Theorem~4.4.2]{BB04} on the following diagram.
	\[
		\begin{tikzcd}
			0 \arrow[r] & {[\Ker(s_G+s_X), \Ker(t_G+t_X)]} \arrow[d, nmono]\arrow[r, equal] & {[\Ker(s_G+s_X), \Ker(t_G+t_X)]} \arrow[d, nmono]\arrow[r] & 0 \arrow[d, nmono] \arrow[r] & 0 \\
			0 \arrow[r] & \Ker(\CoindArr{f\and 1_X})_\Grp \arrow[r, nmono]                  & G+X\arrow[r, "\CoindArr{f\and 1_X}"', repi]                & X \arrow[r]                  & 0
		\end{tikzcd}
	\]
\end{remark}

\subsection{The obstruction and open questions}\label{sec - chi problem with XMod}
For \(\PXMod\) the main difficulties to define the embedding were the description of the power object (see \cref{subsubsec - power object in PXMod}) and the description of a situation where the group morphism
\[\chi\colon \Ker(\CoindArr{f\and 1_X})\to A \mapping (g,x)\mapsto \theta(x)\cdot g\cdot \theta(x\cdot f(g))^{-1}\]
becomes a morphism of precrossed modules (see \cref{sec - suitable sections in PXMod}). For \(\XMod\), the description of the power object and the application of Gray's construction to complete the definition of the left adjoint \(R\) is not very difficult. The main issue is about the morphism \(\chi\). Indeed, its domain is now a quotient (see \cref{thm - kernel in XMod final version}) of the one in groups/precrossed modules. In particular, the neutral element in the quotient must be sent by \(\chi\) to the neutral element in \(A\) but this seems impossible without any further assumption on the reflexive graph section \(\theta\).

Recall that when we define \(\chi\) for \(\Grp\) and \(\PXMod\), we use the presentation \(\langle G\times X\mid R\rangle\), where \(R\) is as in \eqref{eq: relation for kernel in group}. In this presentation, an element \((g,x)\in G\times X\) is seen as the word \(xgf(g)^{-1}x^{-1}\) in \(G+X\).

Let us understand the commutator \([\Ker(s_G+ s_X) , \Ker(t_G+ t_X)]\) in terms of generators \(G\times X\) (since we use this presentation to define \(\chi\) in the previous cases). Consider, for instance, the element
\begin{equation*}
	\bigl(xs_X(x)^{-1}\bigr)\bigl(ht_G(h)^{-1}\bigr)\bigl(s_X(x)x^{-1}\bigr) \bigl(t_G(h)h^{-1}\bigr)\in [\Ker(s_G+ s_X) , \Ker(t_G+ t_X)]
\end{equation*}
where \(x\in X\) and \(h\in G\). Using the same technique as in the beginning of \cref{subsubsection KLE in XMod}, we can rewrite this element as
\begin{align*}
	\bigl(xs_X(x)^{-1}\bigr)\bigl(ht_G(h)^{-1}\bigr)\bigl(s_X(x)x^{-1}\bigr) \bigl(t_G(h)h^{-1}\bigr) & =\omega_1 x_1 t_G(h)h^{-1}=\omega_1 f(h)t_X(f(h))^{-1}t_G(h)h^{-1}                       \\
	                                                                                                  & =\omega_1\omega_2 \left(f(h)t_X(f(h))^{-1}f(t_G(h))f(h)^{-1}\right)    =\omega_1\omega_2
\end{align*}
where
\begin{itemize}
	\item \(\omega_1=xs_X(x)^{-1}ht_G(h)^{-1}f(t_G(h))f(h)^{-1}s_X(x)x^{-1}\in \Ker(\CoindArr{f\and 1_X})_{\Grp}\)
	\item since \(xs_X(x)^{-1}f(h)t_X(f(h))^{-1}s_X(x)x^{-1}t_X(f(h))f(h)^{-1}\in[\Ker(s_X),\Ker(t_X)]=\{e_X\}\), we have \(x_1=f(h)t_X(f(h))^{-1}\);
	\item \(\omega_2=f(h)t_X(f(h))^{-1}t_G(h)h^{-1}f(h)f(t_G(h))^{-1}t_X(f(h))f(h)^{-1}\in \Ker(\CoindArr{f\and 1_X})_{\Grp}\).
\end{itemize}

As a result, in terms of the generators of the presentation, the word \(\omega_1\omega_2\) may be understood as
\begin{equation}\label{eq - commutator in terms of generator for chi}
	(ht_G(h)^{-1}, x s_X(x)^{-1})(t_G(h)h^{-1}, f(h)t_X(f(h))^{-1})\in G\times X\text{.}
\end{equation}
If we consider the same \(\chi\) as for \(\Grp\) and \(\PXMod\) and want it to factorise through our quotient, then the image of~\eqref{eq - commutator in terms of generator for chi} via \(\chi\) must be trivial. However, this is not the case since we have
\begin{align*}
	 & \chi\bigl((ht_G(h)^{-1}, x s_X(x)^{-1})(t_G(h)h^{-1}, f(h)t_X(f(h))^{-1})\bigr)                                     \\
	 & =\left(\theta(xs_X(x)^{-1})\cdot ht_G(h)^{-1}\cdot \theta(xs_X(x)^{-1}\cdot f(ht_G(h)^{-1}))^{-1}\right)            \\
	 & \left(\theta(f(h)t_X(f(h))^{-1})\cdot t_G(h)h^{-1}\cdot \theta(f(h)t_X(f(h))^{-1}\cdot f(t_G(h)h^{-1}))^{-1}\right)
\end{align*}
which turns to be equal to the neutral element only if the reflexive graph section \(\theta\) is also a group morphism, i.e.\ a section of (pre)crossed module for \(f\).

To see this, set \(u\eqdef xs_X(x)^{-1}\in \Ker(s_X)\) and \(v\eqdef f(h)t_X(f(h))^{-1}\in \Ker(t_X)\), so that \(f(ht_G(h)^{-1})=v\). Writing \(a\eqdef ht_G(h)^{-1}\theta(v)^{-1}\in A\) and letting \(\beta(u,v)\eqdef \theta(u)\theta(v)\theta(uv)^{-1}\in A\) be the multiplicativity defect of the section, the expression above becomes
\[
	\bigl(\theta(u)\,a\,\theta(u)^{-1}\bigr)\,\beta(u,v)\,a^{-1}\,\theta(e)^{-1}\text{.}
\]
When \(\theta\) is a group morphism we have \(\beta(u,v)=e\) and \(\theta(e)=e\), so this reduces to the commutator \([\theta(u),a]\); and since \(\theta(u)\in \Ker(s_G)\) and \(a\in \Ker(t_G)\), it vanishes because \([\Ker(s_G),\Ker(t_G)]=\{e\}\) in the crossed module \(G\). Thus the Peiffer condition on \(G\) takes care of the conjugation part, and the only genuine obstruction is the multiplicativity defect \(\beta(u,v)\), which survives in \(A\) whenever \(\theta\) fails to be multiplicative.

From the above discussion, it is clear that we cannot follow the same strategy as we did for \(\PXMod\): if we try to do so, we would have to restrict ourselves to split extensions, which is a much stricter condition than what we asked for precrossed modules and takes us back to the fact that \(\XMod\) is \LACC{}\@, a result already well known.

\section*{Acknowledgments}
Many thanks to Tim Van der Linden, our PhD supervisor, for bringing us together on a topic that combines our areas of expertise and our thanks as well for the discussions and feedback that enhanced this document.

\printbibliography    

\end{document}